\documentclass[reqno,12pt,letterpaper]{amsart}
\usepackage[proof]{macros}
\usepackage{enumerate}
\usepackage{subfig,caption}
\usepackage{graphicx}				
\usepackage{amssymb}
\usepackage{amsmath}
\usepackage{amsthm,amsfonts,bm,bbm}
\usepackage{mathtools}

\usepackage{tikz}
\usetikzlibrary{patterns}
\usepackage{pgfplots}
\pgfplotsset{width=7cm,compat=1.8}

\mathtoolsset{showonlyrefs=true}
\usepackage{mathrsfs}
\usepackage{xcolor}
\usepackage{enumerate,enumitem}
\usepackage[inline]{}
\usepackage{leftidx}
\numberwithin{equation}{section}
\usepackage{marginnote}
\title[Sharp polynomial decay for singular damping]{Sharp polynomial decay for polynomially singular damping on the torus}
\author{Perry Kleinhenz}
\address{Department of Mathematics, Michigan State University, East Lansing, MI 48824, USA}
\email[P.~Kleinhenz]{kleinh29@msu.edu}

\author{Ruoyu P. T. Wang}
\address{Department of Mathematics, Northwestern University, Evanston, IL 60208, USA}
\email[R.~P.~T.~Wang]{rptwang@math.northwestern.edu}

\keywords{damped waves, singular damping, backward uniqueness, Schr\"odinger observability}

\newcommand{\Rb}{\mathbb{R}}

\newcommand{\C}{\mathbb{C}}

\newcommand{\Dc}{\mathcal{D}}

\newcommand{\Nb}{\mathbb{N}}

\newcommand{\essinf}{\mathop {\rm ess\,inf}}

\newcommand{\ra}{\rightarrow}

\renewcommand{\>}{\right\rangle}
\newcommand{\e}{\varepsilon}
\renewcommand{\d}{\delta}

\newcommand{\nm}[1]{\left\| #1 \right\|}

\newcommand{\lp}[2]{ \nm{#1}_{L^{#2}}}

\newcommand{\hp}[2]{\nm{#1}_{H^{#2}}}

\newcommand{\ltwo}[1]{\lp{#1}{2}}

\newcommand{\Cs}{C^{\infty}_0}

\newcommand{\p}{\partial}

\renewcommand{\supp}{\text{supp }}
\newcommand{\T}{\mathbb{T}}
\newcommand{\Ac}{\mathcal{A}}
\newcommand{\Lc}{\mathcal{L}}

\newcommand{\Hc}{\mathcal{H}}

\newcommand{\ti}{\widetilde}

\newcommand{\Sb}{\mathbb{S}}

\renewcommand{\Re}{\text{Re }}

\newcommand{\abs}[1]{\left|#1\right|}

\newcommand{\cim}{\operatorname{Im}}
\newcommand{\cre}{\operatorname{Re}}
\newcommand{\id}{\operatorname{Id}}

\newcommand{\Acd}{\dot{\Ac}}
\newcommand{\Hcd}{\dot{\Hc}}
\renewcommand{\ker}{\operatorname{Ker}}
\newcommand{\Pl}{P_{\lambda}}
\newcommand{\Pli}{P_{\lambda}^{-1}}

\renewcommand{\Dc}{\mathcal{D}}
\newcommand{\Dcd}{\dot{\Dc}}

\newcommand{\spec}{\operatorname{Spec}}
\newcommand{\bigo}{\mathcal{O}}
\renewcommand{\sgn}{\operatorname{sgn}}
\newcommand{\spa}{\operatorname{span}}
\newcommand{\av}{\operatorname{Av}}

\begin{document}

\begin{abstract}
We study energy decay rates for the damped wave equation with unbounded damping, without the geometric control condition. Our main decay result is sharp polynomial energy decay for polynomially controlled singular damping on the torus. We also prove that for normally $L^p$-damping on compact manifolds, the Schr\"odinger observability gives $p$-dependent polynomial decay, and finite time extinction cannot occur. We show that polynomially controlled singular damping on the circle gives exponential decay. 
\end{abstract}
\maketitle

\section{Introduction}
\subsection{Introduction}
In this paper we study the damped wave equation. Let $(M,g)$ be a compact Riemannian manifold without boundary and let $W$ be a non-negative measurable function on $M$. Then the viscous damped wave equation is
\begin{equation}\label{DWE}
\begin{cases}
(\p_t^2 -\Delta + W \p_t ) u =0, \text{ in } M\\
(u, \p_t u)|_{t=0} = (u_0, u_1) \in \Dc\subset H^1(M)\times L^2(M), 
\end{cases}
\end{equation}
where the set of admissible initial data $\Dc$ is specified later in Definition \ref{1l4}. The primary object of study in this paper is the energy 
\begin{equation}
E(u,t) = \frac{1}{2} \int \abs{\nabla u}^2 +  \abs{\p_t u}^2 \ dx.
\end{equation}
When $W$ is continuous, it is classical that uniform stabilization is equivalent to geometric control by the positive set of the damping. That is, there exists $r(t) \ra 0$ as $t \ra \infty$ such that
\begin{equation}
E(u,t) \leq C r(t) E(u,0),
\end{equation}
if and only if there exists $L$, such that all geodesics of length at least $L$ intersect $\{W>0\}$. In this case, due to the semigroup property of solutions, one can take $r(t)=\exp({-ct})$, for some $c>0$.  

When the geometric control condition is not satisfied, $E(u,0)$ must be replaced on the right hand side. Decay rates are of the form
\begin{equation}\label{e:endec}
E(u,t)^{1/2} \leq C r(t) \left(\hp{u_0}{2} + \hp{u_1}{1}  \right).
\end{equation}
Furthermore, the optimal $r(t)$ depends not only on the geometry of $M$ and $\{W>0\}$, but also on properties of $W$ near $\partial\{W>0\}$. In general, for bounded $W$ the more singular $W$ is near $\p \{W>0\}$, the slower the sharp energy decay rate. Since unbounded damping allows for even more singular behavior it is natural to see if this relationship continues. 

For much of this paper we focus on $M=\T^2$, which has polynomial rates $r(t)=\langle t\rangle^{-\alpha}$ with $\alpha>0$. In particular, when $W \in L^{\infty}(M)$, is positive on a positive measure set, but $\supp W$ does not satisfy the geometric control condition, \cite{aln14} showed that \eqref{e:endec} holds with $r(t)=\langle t\rangle^{-\frac{1}{2}}$ and cannot hold with $r(t)=\langle t \rangle^{-1-\e}$.  Although there are damping functions which saturate the $\langle t\rangle^{-1}$ rate, there are no examples of damping for that the waves must decay at $\langle t\rangle^{-\frac{1}{2}}$. The closest rate was proved by Nonnenmacher in an appendix to \cite{aln14}: for damping equal to the characteristic function of a strip, there are solutions decaying no faster than $\langle t\rangle^{-\frac{2}{3}}$. This leaves a mysterious gap between $\langle t\rangle^{-\frac{1}{2}}$ and $\langle t\rangle^{-\frac{2}{3}}$ where no known bounded damping gives such a rate.

In this paper, we aim to address two open problems formulated in \cite{aln14}: 
\begin{question}\label{p1}
Is the a priori upper bound $\langle t\rangle^{-\frac{1}{2}}$ for the rates optimal?
\end{question}
\begin{question}\label{p2}
How is the vanishing rate of $W$ related to the energy decay rate?
\end{question}
In order to do this, we no longer assume $W\in L^{\infty}$ and consider unbounded damping. 

\subsection{Main results}
Now we list our main results in this paper. First we show that on $M=\mathbb{T}^2$, for $\beta\in (-1,\infty),$ $y$-invariant damping functions vanishing like $d(\cdot, \{W=0\})^{\beta}$ near $\partial\{W=0\}$ produce sharp $\langle t\rangle^{-\frac{\beta+2}{\beta+3}}$-decay. This interpolates the decay rates between $\langle t\rangle^{-\frac{1}{2}}$ and $\langle t\rangle^{-1}$. The possibly unbounded polynomially controlled damping functions in $X_\theta^\beta$ and the space $\mathcal{D}$ of admissible initial data are specified in Definitions \ref{1l5} and \ref{1l4} below. 
\begin{theorem}[Polynomial decay on the torus]\label{torusthm}
Let $M =\T^2$ and fix $n \in \Nb$ and $\beta \in(-1,\infty)$. Suppose $W(x,y) = \sum_{j=1}^n W_j(x)$, with $W_j \in X^{\beta}_{\theta_j}$. Then there exists $C>0$ such that for all $t> 0$, 
\begin{equation}\label{e:polydecay}
E(u,t)^{\frac{1}{2}} \leq C\langle t\rangle^{-\frac{\beta+2}{\beta+3}} \left(\|\nabla u_0\|_{L^2}+\|u_1\|_{H^1}+\|-\Delta u_0+Wu_1\|_{L^2}\right),
\end{equation}
for the solution $u$ of \eqref{DWE} with respect to any initial data $(u_0, u_1) \in \Dc$. 
\end{theorem}
\begin{theorem}[Sharpness of the decay rates]\label{sharpnessthm}
The rates obtained in Theorem \ref{torusthm} are sharp for $\beta\in(-1,\infty)$. Indeed, let $W(x,y)=\mathbbm{1}_{\abs{x}\ge \frac{\pi}{2}}\left((\abs{x}-\frac{\pi}{2})^\beta+1\right)$ for $\beta\in (-1,0)$. Then for all $C,\epsilon>0$, there exists $t_0>0$ and initial data $(u_0,u_1) \in \Dc$ such that
\begin{equation}
E(u,t_0)^{\frac{1}{2}} > C\langle t_0\rangle^{-\frac{\beta+2}{\beta+3}-\epsilon} \left(\|\nabla u_0\|_{L^2}+\|u_1\|_{H^1}+\|-\Delta u_0+Wu_1\|_{L^2}\right),
\end{equation}
for the solution $u$ of \eqref{DWE} with respect to the initial data $(u_0, u_1)$. 
\end{theorem}

Theorems \ref{torusthm} and \ref{sharpnessthm} address Question \ref{p2}: we show that damping of more/less singular behaviour than a characteristic function gives slower/faster decay than $\langle t\rangle^{-\frac{2}{3}}$, and give an explicit relation between the decay rates and the vanishing rates of $W$: when $\beta\rightarrow -1^+$, the decay rate approaches $\langle t\rangle^{-\frac{1}{2}-}$, the upper bound for the energy decay obtained in \cite{aln14}; when $\beta\rightarrow \infty$, the decay rate approaches $\langle t\rangle^{-1+}$, the lower bound. This also fills up the mysterious gap between the upper bound $\langle t\rangle^{-\frac12}$ they found and the upper bound $\langle t\rangle^{-\frac23}$ they could produce with bounded damping. However, there are some complications, as we discuss in Theorem \ref{controlthm}.
Note for $\beta \in [0,\infty)$, the decay rates in Theorem \ref{torusthm} were already shown in \cite{dk20,sta17} and their corresponding sharpness results in Theorem \ref{sharpnessthm} were shown in \cite{kle19b,aln14}, so we here only need to address the case $\beta \in (-1,0)$. Now we define the terminologies:

\begin{figure}
\centering
\subfloat[$V^\beta$ on $\mathbb{S}^1$]{
\includegraphics[page=1, height=10em]{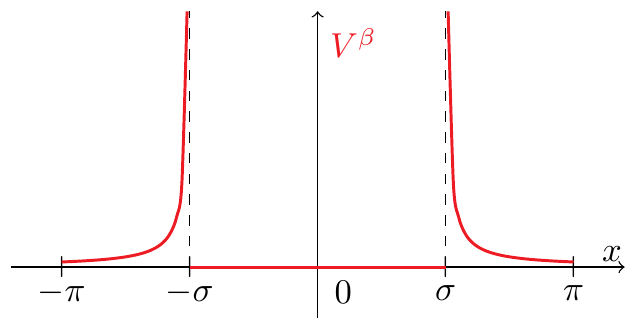}

}\hfill
\subfloat[$W(x)\in X^\beta_\pi$ on $\mathbb{T}^2$]{
\includegraphics[height=10em]{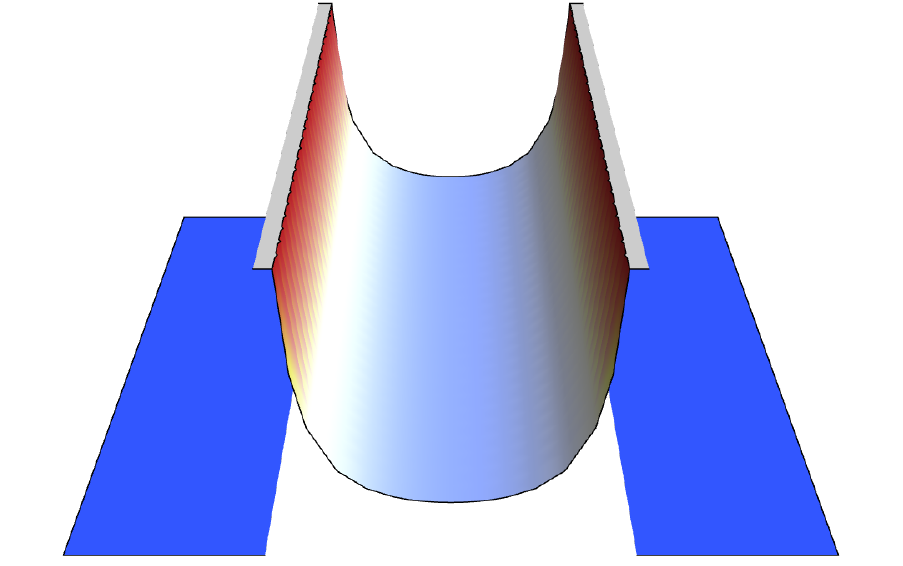}

}
\caption{Polynomial control functions $V^\beta$ and typical damping $W$. }
\label{f1}
\end{figure}

\begin{definition}[Polynomially controlled damping]\label{1l5}
Fix $\theta\in \mathbb{S}^1$. Parametrize $\Sb^1$ by $[-\pi, \pi)$ with ends identified periodically and $\theta=0$. We define, $X_\theta^\beta$, the space of polynomially controlled functions on $\Sb^1$ as the space of measurable functions $f$ on $\mathbb{S}^1$ such that there are $C_0, \sigma>0$ such that $C_0^{-1} V^{\beta}(x) \leq f(x) \leq C_0 V^{\beta}(x)$, where
\begin{equation}\label{1l3}
V^{\beta}(x) = \begin{cases} 0, \qquad &x \in [-\sigma,\sigma],\\
(|x|- \sigma)^\beta, \qquad & |x| \in (\sigma, \pi).
\end{cases} 
\end{equation}
\end{definition}

We will mainly consider damping $W(x,y)$ that is a sum of $V^\beta(x)$ with $\beta \in (-1,0)$ in this paper, so $W$ locally blows up like the polynomial $(\abs{x}-\sigma)^{\beta}$ near $\pm \sigma$: see Figure \ref{f1}. Because of the unboundedness of the damping $W$ we need to adjust the space our initial data is in. 

\begin{definition}[Admissible initial data]\label{1l4}
Let $\Dc=\{(u_0,u_1)\in H^1(M)\times H^1(M): -\Delta u_0+W u_1\in L^2(M)\}$ equipped with the norm
\begin{equation}
\|(u_0,u_1)\|_\Dc^2=\|u_0\|_{H^1}^2+\|u_1\|_{H^1}^2+\|-\Delta u_0+Wu_1\|_{L^2}^2. 
\end{equation}
\end{definition}
\begin{remark}[Relation between $\Dc$ and $H^2\times H^1$]
We remark that when $\beta\in (-\frac{1}{2},\infty)$, $\Dc=H^2(M)\times H^1(M)$ and the $\Dc$-norm is equivalent to that of $H^2\times H^1$. When $\beta\in (-1, -\frac{1}{2}]$,
\begin{equation}
H^{\frac{3}{2}-} \times H^1 \supset \Dc \supsetneq H^2 \times \{u_1 \in H^1: u_1=0 \text{ on } W^{-1}(\infty)\}.
\end{equation}
Furthermore, in this case, $\Dc$ and $H^2\times H^1$ are mutually not a subset to each other. 
\end{remark}

We now move on to results in the setting of normally $L^p$-damping on compact manifolds. Normally $L^p$-damping are those damping functions that look like an $L^p$-function along the normal direction near a hypersurface, given in Definition \ref{2t14}. Note that for $\beta \in (-1,0)$, polynomially controlled damping in $X^{\beta}_{\theta}$, are normally $L^{-\frac{1}{\beta}-}$ and for $\beta\ge 0$ they are (normally) $L^{\infty}$.

The natural question to ask, is if the a priori upper bound $\langle t\rangle^{-\frac12}$, obtained via Schr\"odinger observability, still holds when $W\notin L^\infty$? The answer is negative: although we are able to obtain an a priori decay rate for such $W$ via Schr\"odinger observability, unfortunately it is slower.

\begin{theorem}[Schr\"odinger observability gives polynomial decay]\label{controlthm}
Let $W$ be normally $L^p$ for $p\in (1,\infty)$. Assume the Schr\"odinger equation is exactly observable from an open set $\Omega$, and there exists $\epsilon>0,$ such that $W\ge \epsilon$ almost everywhere on an open neighbourhood of $\overline{\Omega}$. Then 
\begin{equation}
E(u,t)^{\frac{1}{2}} \leq C\langle t\rangle^{-\frac{1}{2+\frac1p}} \left(\|\nabla u_0\|_{L^2}+\|u_1\|_{H^1}+\|-\Delta u_0+Wu_1\|_{L^2}\right),
\end{equation}
for the solution $u$ of \eqref{DWE} with respect to any initial data $(u_0, u_1) \in \Dc$. When $p=1$, $E^{\frac{1}{2}}\le C\langle t\rangle^{-1/(3+)}$. When $p=\infty$, $E^{\frac{1}{2}}\le C\langle t\rangle^{-1/2}$ . 
\end{theorem}
When $p=\infty$, we can indeed have the $\langle t\rangle^{-1/2}$-decay: this bounded case was known in \cite{aln14}. Together with the Schr\"odinger observability results in \cite{bz12,am14}, we immediately have:
\begin{corollary}
Suppose $M=\mathbb{T}^d$ for $d\ge 1$. Let $W$ be normally $L^p$ for $p\in (1,\infty)$, and assume there exists $\epsilon>0$, such that $W\ge \epsilon$ almost everywhere on an open set. Then 
\begin{equation}
E(u,t)^{\frac{1}{2}} \leq C\langle t\rangle^{-\frac{1}{2+\frac1p}} \left(\|\nabla u_0\|_{L^2}+\|u_1\|_{H^1}+\|-\Delta u_0+Wu_1\|_{L^2}\right),
\end{equation}
for the solution $u$ of \eqref{DWE} with respect to any initial data $(u_0, u_1) \in \Dc$. When $p=1$, $E^{\frac{1}{2}}\le C\langle t\rangle^{-1/(3+)}$. When $p=\infty$, $E^{\frac{1}{2}}\le C\langle t\rangle^{-1/2}$ . 
\end{corollary}

Theorems \ref{torusthm}, \ref{sharpnessthm} and \ref{controlthm} give both positive and negative answers to Question \ref{p1}: on one hand, in Theorem \ref{sharpnessthm} for any polynomial rates between $\langle t\rangle^{-\frac{1}{2}}$ and $\langle t\rangle^{-\frac{2}{3}}$, we found singular damping functions that are sharply stabilized at that rate. On the other hand, Theorem \ref{controlthm} implies that for normally $L^p$-damping, the upper bound obtained via the Schr\"odinger observability can be as weak as $\langle t\rangle^{-\frac{1}{3}+}$ as $p\rightarrow 1^+$, instead of $\langle t\rangle^{-\frac{1}{2}}$. Thus while we filled the gap between $\langle t\rangle^{-\frac{1}{2}}$ and $\langle t\rangle^{-\frac{2}{3}}$ by making $W$ more singular, we also create a new gap between $\langle t\rangle^{-\frac{1}{3}}$ and $\langle t\rangle^{-\frac{1}{2}}$.

A peculiar feature to note is that observability of the Schr\"odinger equation does not depend on the singularity of $W$, but the decay rate produced does. Such dependency is not observed in the case of bounded damping. We point out that the rates we obtain in Theorem \ref{controlthm} are better than the rates obtainable with \cite{cpsst19}: see Remark \ref{2t16}. 

We now move on to address some open problems concerning the finite time extinction phenomenon concerning singular damping.
\begin{theorem}[Backward uniqueness]\label{backwardthm}
Let $W$ be normally $L^p$ for $p\in (1,\infty]$. Then the damped wave semigroup $e^{t\Ac}$ is backward uniqueness. That is:
\begin{enumerate}
\item If $u(t)=0, \partial_t u(t)=0$ at some $t>0$, then $u_0=u_1=0$. 
\item If $E(u,t)=0$ at some $t>0$, then $E(u,0)=0$. 
\end{enumerate}
 
\end{theorem}
This theorem states that there cannot be finite-time extinction of solutions or energy when the damping vanishes like $x^\beta$ for $\beta\in (-1,\infty]$. This is in contrast to the limit case $\beta=-1$, studied in \cite{fhs20,cc01} where particular setups with damping of the form $2x^{-1}$ were found and all solutions go extinct in finite time. Such finite-time extinction phenomenons are of note as they are rarely observed for linear equations. 

On $M=\Sb^1$, our last theorem states that exponential decay occurs when $W$ is a finite sum of polynomially controlled functions and bounded functions. 
\begin{theorem}[Exponential decay on the circle]\label{circlethm}
Suppose $M=\Sb^1$ and for $1 \leq j \leq n,$ $W_j \in L^{\infty}$ or $W_j \in X_{\theta_j}^{\beta_j}$ with $\beta_j \in (-1, 0)$. If $W=\sum_{j=1}^n W_j$, then there exist $C,c>0$ such that 
\begin{equation}
E(u,t) \leq Ce^{-ct} E(u,0),
\end{equation}
for the solution $u$ of \eqref{DWE} with respect to any initial data $(u_0, u_1) \in H^1\times L^2$. 
\end{theorem}

This theorem shows that, in one dimension, the geometric control implies exponential decay even if there are some singularities, as long as the singularities are not too large. 

\begin{remark}
We heuristically interpret that the singularity of $W$ near $\partial\{W=0\}$ prevents the propagation of low-frequency modes into $\{W>0\}$.  The singularity reflects energy back into $\{W=0\}$ as well as transmitting some into $\{W>0\},$ and the greater the singularity the more energy it reflects. In the setting of Theorems \ref{torusthm} and \ref{sharpnessthm}, the low-frequency sideways propagation of vertically concentrated modes in $\{W=0\}$ has a $\beta$-dependent transmission rate into $\{W>0\}$. This is why the energy decay rate depends on $\beta$. On the other hand, the singularity does not affect high-frequency propagation much. This is best seen in the setting of Theorem \ref{circlethm}, where all high-frequency modes penetrate into $\{W>0\}$ and are damped exponentially.
\end{remark}

\begin{remark}
\begin{enumerate}[wide]
\item Theorems \ref{torusthm} and \ref{circlethm} hold for the limit case $\beta=-1$, when the evolutional semigroup is restricted to $\{(u,v)\in H^1\times L^2: u|_{W^{-1}(\infty)}=0\}$ with its generator defined on $\{(u,v)\in H^2\times H^1: u|_{W^{-1}(\infty)}=v|_{W^{-1}(\infty)}=0\}$, without any major changes to the proofs. To keep this paper concise, we choose not to pursue this direction. 
\item Note, the proof of Theorem \ref{circlethm} suffices for $W_j= 1_{[-\pi,-\sigma]}(x) V_j$ or $1_{[\sigma,\pi)}(x)V_j$ for $V_j \in X^{\beta_j}_{\theta_j}$. However for ease of notation we only work with the symmetric definition of polynomial control. See Remark \ref{oneside} for further details. 
\end{enumerate}
\end{remark}

\subsection{Literature review}
The equivalence of uniform stabilization and geometric control for continuous damping functions was proved by Ralston \cite{ral69}, and Rauch and Taylor \cite{rt74} (see also \cite{blr88}, \cite{blr92} and \cite{bg97}, where $M$ is also allowed to have a boundary). For finer results concerning discontinuous damping functions, see Burq and G\'erard \cite{bg20}.

Decay rates of the form \eqref{e:endec} go back to Lebeau \cite{leb93}. If we assume only that $W \in C(M)$ is non-negative and not identically $0$, then the best general result is that $r(t)=1/\log(2+t)$ in \eqref{e:endec} \cite{bur98},\cite{leb93}. Furthermore, this is optimal on spheres and some other surfaces of revolution \cite{leb93}. For recent work on logarithmic decay see \cite{bm23}. At the other extreme, if $M$ is a negatively curved (or Anosov) surface, $W \in C^\infty(M)$, and $W \not \equiv 0$, then $r(t)$ may be chosen exponentially decaying in \cite{dj18,jin20,djn22}. 

When $M$ is a torus, these extremes are avoided and the best bounds are polynomially decaying in \eqref{e:endec}. Anantharaman and L\'eautaud \cite{aln14} show \eqref{e:endec} holds with $\<t\>^{-1/2}$ when $W \in L^\infty$, and $W\ge \epsilon$ on some open set for some $\epsilon>0$, as a consequence of Schr\"odinger observability/control  \cite{jaf90, mac10, bz12, am14}. The more recent result of Burq and Zworski on Schr\"odinger observability and control \cite{bz19} weakens the final requirement  to merely $W \not \equiv 0$. Anantharaman and L\'eautaud \cite{aln14} further show that if $\supp W$ does not satisfy the geometric control condition then \eqref{e:endec} cannot hold for $\<t\>^{-1-\e}$. They also show if there exists $C>0$, such that $W$ satisfies $|\nabla W| \leq C W^{1-\e}$ for $\e<1/29$, and $W \in W^{k_0, \infty}$ for $k_0 \geq 8$, then $\eqref{e:endec}$ holds with $\<t\>^{\frac{-1}{1+4\e}}$. See also \cite{bh07}. 

Sharp decay results have been obtained on the torus when the damping is taken to be polynomially controlled, bounded and $y$-invariant. In particular \cite{kle19b} and \cite{dk20} together show that for such damping with $\beta> 0$ \eqref{e:endec} holds with $\<t\>^{-\frac{\beta+2}{\beta+3}}$ and there are some solutions decaying no faster than this rate. See also \cite{aln14} and \cite{sta17} for the original proof of the case $\beta=0$. For improved decay rates under different geometric assumptions on the support of the damping see \cite{ll17} and \cite{sun22}. In \cite{wan21,wan21b}, the second author showed that \eqref{e:endec} holds with $\<t\>^{-\frac{1}{2}}$ when there is boundary damping. The boundary damping has a singularity structure similar to $\delta(x)$, which is in the Besov space $B^{-1}_{\infty,\infty}$, hinting that the decay rate $\<t\>^{-\frac{\beta+2}{\beta+3}}$ still holds when $\beta=-1$. We then conjectured that $\<t\>^{-\frac{\beta+2}{\beta+3}}$ holds and is optimal for all $\beta\in (-1,\infty)$. In this paper, we prove that our conjecture is correct. 

As mentioned above, using observability of an associated Schr\"odinger equation to prove energy decay for the damped wave equation was used in \cite{aln14} to prove $\<t\>^{-\frac{1}{2}}$ energy decay. This relied on a characterization of observability due to \cite{mil05}. This approach was also applied to prove energy decay for a semilinear damped wave equation in \cite{jl20}. A more abstract, semigroup focused treatment for singular damping is provided in \cite{cpsst19}.

Another motivation for the study of unbounded damping to understand their overdamping behavior. Overdamping here heuristically means ``more'' damping leads to slower decay. At a basic level this can be observed on $M=[0,1]$ with $W=C$, a constant. Such a damping satisfies the GCC and so experiences exponential decay, but as shown in \cite{cz94} the exponential rate is not monotone in $C$. The decay becomes faster as $C$ increases from $0$, up to a point, and then becomes slower as $C$ goes to infinity. Because of this, one might expect that unbounded damping would exhibit slower decay than bounded analogs. However, when $M=[0,1]$ and $W=\frac{2}{x},$  \cite{cc01} show that all solutions are identically $0$ for $t>2$. The case where $W=\frac{2\alpha}{x}$ on $[0,1]$ for $\alpha$ a constant was studied in \cite{fhs20}. The authors showed that although the finite extinction time behavior is unique to $\alpha=1$, in general all solutions decay exponentially and for $\alpha \in \Nb$ most solutions have a finite extinction time. See also \cite{maj74}. Unbounded damping has also been studied on non-compact manifolds in \cite{fst18}, \cite{Gerhat2022} and \cite{Arnal2022}.

\subsection{Paper outline} 
In Section \ref{semigroupsection}, we rigorously formulate the strongly continuous damped wave semigroup $e^{t\Ac}:\Hc\rightarrow\Hc$ on $\Hc=H^1\times L^2$. The stability of $e^{t\Ac}$ is related to the $L^2$-resolvent estimates for $P_\lambda=-\Delta-i\lambda W-\lambda^2$, the family of stationary damped wave operators in $\lambda\in \mathbb{C}$. This is characterised by the next proposition: 
\begin{proposition}[Equivalence between resolvent estimates and decay]\label{polyprop}
\begin{enumerate}[wide]
The following are true:
\item Fix $\alpha>0$. There exists $C>0$ such that
\begin{equation}
E(u,t)^{1/2} \leq C\langle t\rangle^{-\alpha}\left(\|\nabla u_0\|_{L^2}+\|u_1\|_{H^1}+\|-\Delta u_0+Wu_1\|_{L^2}\right),
\end{equation}
for all solutions $u$ with initial data $(u_0,u_1) \in D$, if and only if, there exists $C>0$ such that for all $\lambda\in \mathbb{R}$ and $\lambda\neq0$ we have $\| \Pli \|_{L^2 \ra L^2} \leq C |\lambda|^{1/\alpha -1}$.
\item There exist $C, c>0$ such that $E(u,t) \leq C e^{-ct} E(u,0)$, 
for all solutions $u$ with initial data $(u_0,u_1) \in H^1\times L^2$, if and only if there exist $C>0$ such that for all $\lambda\in \mathbb{R}$ and $\lambda\neq 0$ we have $\| \Pli \|_{L^2 \ra L^2} \leq C/\abs{\lambda}$.
\end{enumerate}
\end{proposition}
Here, the exponential decay resolvent estimate result can be thought of as a refinement of the polynomial result when $\alpha\rightarrow\infty$. We show there are pole-free regions of $P_\lambda$ in Propositions \ref{maplemma} and \ref{2t6}, and use them to prove Theorem \ref{backwardthm}, that $e^{t\Ac}$ is backward unique. We also prove Theorem \ref{controlthm} that the Schr\"odinger observability gives polynomial decay. 

In Section \ref{resolventsection}, we prove the necessary resolvent estimates for $P_\lambda$ using a Morawetz multiplier method. We then use Proposition \ref{polyprop} to prove Theorem \ref{torusthm} giving polynomial decay on $\mathbb{T}^2$, and Theorem \ref{circlethm} giving exponential decay on $\mathbb{S}^1$. In Section \ref{sharpnesssection}, we construct eigenfunctions for the one-dimensional Schr\"odinger operator with a complex Coulomb potential, and quasimodes for the damped wave operator to prove Theorem \ref{sharpnessthm}, the sharpness result.

\subsection{Acknowledgement}
The authors are grateful to Jared Wunsch for many discussions around these results, and to Jeff Galkowski for pointing out an improvement to the Sobolev multiplier estimates in Section \ref{semigroupsection}. The authors are grateful to Romain Joly, Irena Lasiecka, Jeffrey Rauch, Cyril Letrouit and Pedro Freitas for insightful comments. RPTW is partially supported by NSF grant DMS-2054424. 

\section{Semigroups generated by normally \texorpdfstring{$L^p$}{Lp}-damping}\label{semigroupsection}
\begin{figure}
\centering

\subfloat[$N$ relative to $M$.]{\parbox{0.5\linewidth}{\centering \includegraphics[page=5,height=10em]{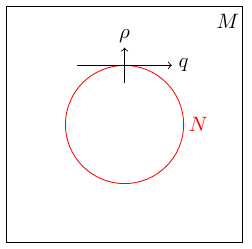}}

}
\subfloat[$N$ relative to $U$.]{\parbox{0.5\linewidth}{\centering\includegraphics[page=6,height=10em]{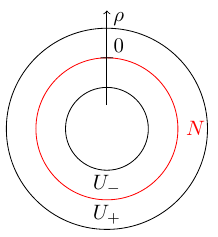}}

}
\caption{The normal structure near the hypersurface $N$ inside $M$.}
\label{f4}
\end{figure}
In this section, we provide a general framework for the damped wave semigroup, with damping $W$, unbounded near a closed hypersurface. We then use this semigroup to prove Proposition \ref{polyprop}.
\subsection{Normally \texorpdfstring{$L^p$}{Lp}-damping and resolvent estimates}
Let $M$ be a compact smooth manifold without boundary. Let $N$ be a closed and orientable hypersurface in $M$, with finitely many components and an orientable normal bundle. Take a normal neighbourhood $U=U_-\sqcup N\sqcup U_+$ divided into two components $U_\pm$ by $N$. Denote $\rho\in C^\infty(U)$ by 
\begin{equation}
\rho(z)=
\begin{cases}
    \pm\operatorname{dist}(z, N), & z\in U_\pm\\
    0, & z\in N. 
\end{cases}
\end{equation}
There exists some small $\delta>0$ such that $N_\delta=\rho^{-1}((-\delta,\delta))$ is compactly embedded in $U$ and $\rho^{-1}(s)$ is diffeomorphic to $N$ for all $s\in (-\delta,\delta)$. We can then identify $N_\delta$ by $N\times (-\delta, \delta)_\rho$: see Figure \ref{f4} for illustration. 

\begin{definition}[Normally $L^p$-damping]\label{2t14}
Assume the damping function $W(z)\ge 0$ on $M$ and is $L^\infty$ on any compact subset of $M\setminus N$. For $p\in [1,\infty]$, we say $W(z)$ is normally $L^p$ (with respect to $N$) if
\begin{equation}
w(\rho)=\operatorname{esssup}\{W(q,\rho): q\in N\}\in L^{p}(-\delta, \delta).
\end{equation}

\end{definition}
The name comes from the fact that $W$ blows up near $N=\rho^{-1}(0)$ like $w(\rho)$, a function $L^p$-integrable along $\rho$, the fiber variable of the normal bundle to $N$. For $1\le p<q\le \infty$, any normally $L^q$-damping is also normally $L^p$. The class of normally $L^\infty$-damping coincide with $L^\infty(M)$. Throughout this section, being normally $L^p$ is the only assumption we impose on $W(z)$: we frequently draw a distinction between functions which are $L^1$ and $L^p$ for $p>1$. We give some important examples of normally $L^1$-damping that we use in other sections:

\begin{figure}
\centering
\subfloat[$W$ represented by the height of the plot on $\mathbb{T}^2$.]{
\parbox{0.45\linewidth}{\centering\includegraphics[height=10em]{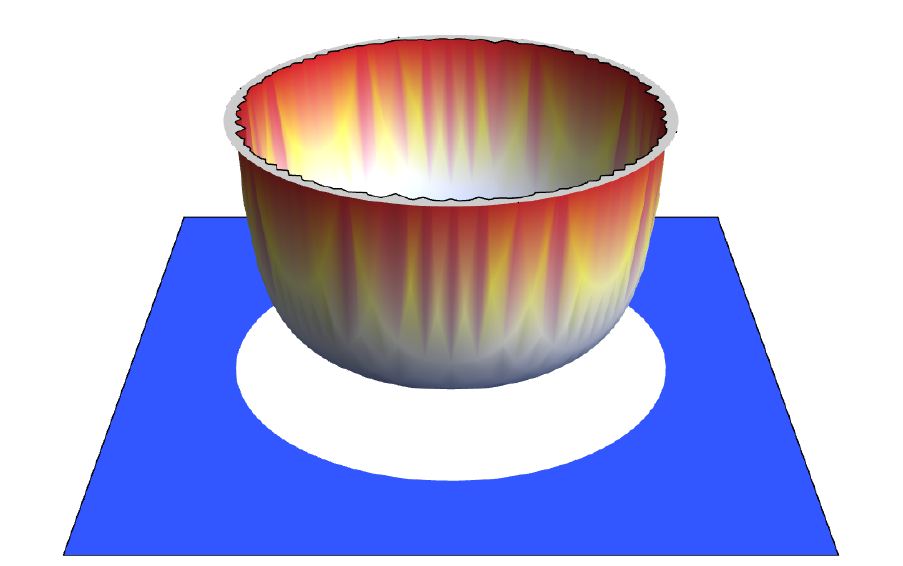}}

}\hfill
\subfloat[$W$ represented by how dark the color is on $\mathbb{S}^2$.]{
\parbox{0.45\linewidth}{\centering\includegraphics[height=10em]{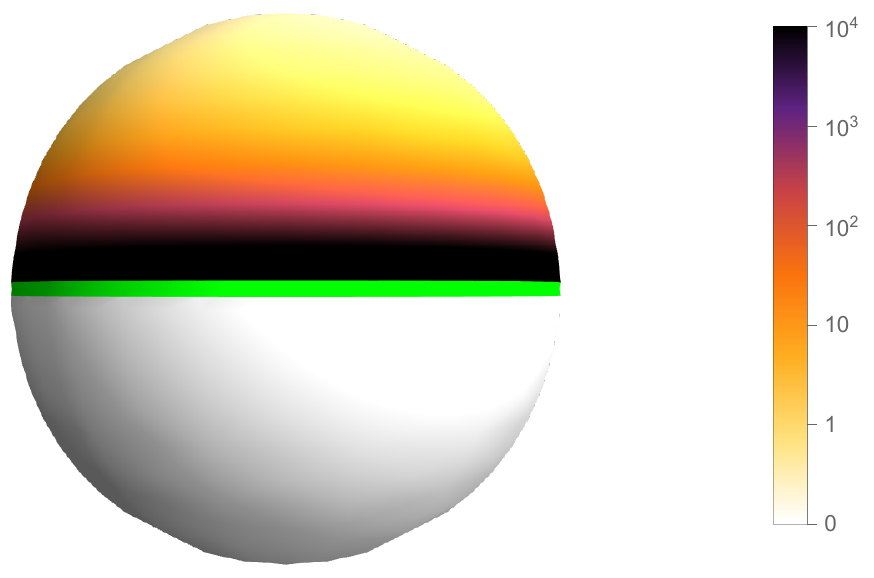}}
}
\caption{Examples of normally $L^p$-damping $W$ on manifolds.}
\label{f2}
\end{figure}

\begin{example}
\begin{enumerate}[wide]
\item Let $M=\mathbb{T}^2$ and $\beta\in (-1,0)$. Then $V^\beta(x)$ defined in \eqref{1l3} are normally $L^{-1/{\beta}-}$ with respect to $\{(x,y)\in \mathbb{T}^2: x=\pm \sigma\}$. They are also normally $L^1$. When $\beta\in [0,\infty)$, they are normally $L^\infty$: see Figure \ref{f1}(B). The same is true for $X^\beta_\theta$. 
\item Let $M=\mathbb{T}^2$, $\beta\in (-1,0)$ and $\epsilon>0$ small. Then $W=\mathbbm{1}_{\{x^2+y^2\le \epsilon^2\}}(\epsilon-(x^2+y^2)^{\frac{1}{2}})^{\beta}$ is normally $L^{-1/\beta-}$ with respect to $\{x^2+y^2=\epsilon^2\}$: see Figure \ref{f2}(A). 
\item Let $M=\mathbb{S}^2$ be equipped with the spherical coordinates $(\theta,\phi)\in [0,2\pi]\times [0,\pi]$. Let $\beta\in (-1,0)$, then $W=\mathbbm{1}_{\{\phi\le \pi/2\}}(\pi/2-\phi)^{\beta}$ is normally $L^{-1/\beta-}$ with respect to $\{\phi=\pi/2\}$, the equator: see Figure \ref{f2}(B).
\end{enumerate}
\end{example}
\begin{lemma}[Sobolev multiplier]\label{2t3}
Let $W$ be normally $L^p$ for $p\in (1,\infty)$. Then the multiplier $\sqrt{W}$ is a bounded map from $H^{\frac{1}{2p}}(M)$ to $L^2(M)$, and extends to a bounded map from $L^2(M)$ to $H^{-\frac{1}{2p}}(M)$. When $p=1$, $\sqrt{W}$ maps from $H^{\frac{1}{2}+}$ to $L^2(M)$ and extends to a bounded map from $L^2(M)$ to $H^{-\frac{1}{2}-}$. When $p=\infty$, $\sqrt{W}$ is bounded on $L^2(M)$.
\end{lemma}
\begin{proof}
1. Let $p\in (1,\infty)$. We show the multiplier $\sqrt{W}$ is bounded from $H^{\frac{1}{2p}}$ to $L^2$. Since $W\in L^\infty(M\setminus N_\delta)$, it suffices to show 
\begin{equation}
\|\sqrt{W}u\|_{L^2(N_\delta)}\le C\|u\|_{H^{\frac{1}{2p}}(N_\delta)}.
\end{equation}
By the Sobolev embedding we have
\begin{equation}
u\in H^{\frac{1}{2p}}(N_\delta)\hookrightarrow L^{\frac{2p}{p-1}}\left( (-\delta,\delta)_\rho, L^2(N)\right).
\end{equation}
Therefore 
\begin{equation}
\|\sqrt{W} u\|_{L^2(N_\delta)}^2\le C \int_{-\delta}^\delta w(\rho)\int_{N} \abs{u(q, \rho)}^2\ dqd\rho \le C\|w\|_{L^{p}_\rho}\|u\|_{L^{\frac{2p}{p-1}}((-\delta, \delta), L^2(N))}
\end{equation}
is bounded by $\|u\|_{H^{\frac{1}{2p}}(N_\delta)}$. 

2. Let $p=1$ and $s>\frac{1}{2}$. By the Sobolev embedding we have
\begin{equation}
u\in H^{s}(N_\delta)\hookrightarrow C^{0,0}\left( (-\delta,\delta)_\rho, L^2(N)\right).
\end{equation}
Therefore 
\begin{equation}
\|\sqrt{W} u\|_{L^2(N_\delta)}^2\le C \int_{-\delta}^\delta w(\rho)\int_{N} \abs{u(q, \rho)}^2\ dqd\rho \le C\|w\|_{L^1_\rho}\|u\|_{L^\infty((-\delta, \delta), L^2(N))}
\end{equation}
is bounded by $\|u\|_{H^{s}(N_\delta)}$. 

3. Let $p=\infty$. Then $W\in L^\infty(M)$ and $\|\sqrt{W}u\|_{L^2(M)}\le C\|u\|_{L^2(M)}$. Thus for all $p\in [1,\infty]$ we have the desired conclusion. 

4. It suffices to observe that for bounded $\sqrt{W}: H^s\rightarrow L^2$, its adjoint $\sqrt{W}^*: L^2\rightarrow H^{-s}$ is bounded as well. 
\end{proof}

\begin{lemma}\label{2t4}
Let $W$ be normally $L^1$. Then for any $u\in L^2(M)$ with $\supp u\subset M\setminus N$, we have $Wu\in L^2(M)$. 
\end{lemma}
\begin{proof}
It suffices to observe that $W$ is essentially bounded on $\supp u$ as a compact subset of $M\setminus N$. Note that we cannot give an uniform multiplier estimate, unless $w(\rho)\in \abs{\rho}^{-1}L^\infty(-\delta,\delta)$. 
\end{proof}

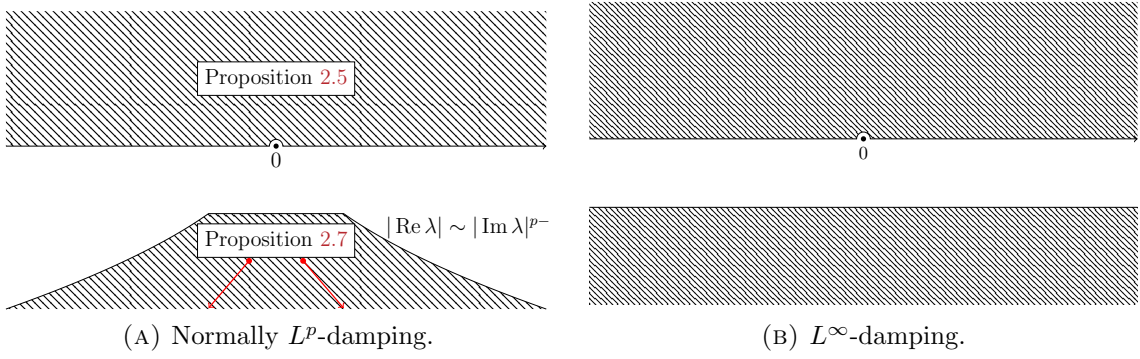
\begin{figure}
\centering
\subfloat[Normally $L^p$-damping.]{
\resizebox{0.48\linewidth}{!}{
\begin{tikzpicture}[scale=2.6, minimum size=0.01cm]
\path [pattern=north west lines] (2,0) -- (2,1) -- (-2,1) -- (-2,0) -- (-0.05, 0) arc (180:0:0.05) -- cycle;
\draw [->] (-2,0) -- (-0.05,0) -- (-0.05, 0) arc (180:0:0.05) -- (2,0); 
\node at (0,0) [circle,fill,inner sep=1pt]{};

\draw plot [domain=2:0.5, smooth, variable=\x] ({-\x},{sqrt(0.5)-0.5-sqrt(\x)}) -- (0.5,-0.5) -- plot [domain=0.5:2, smooth, variable=\x] ({\x},{sqrt(0.5)-0.5-sqrt(\x)});

\path [pattern=north west lines] plot [domain=2:0.5, smooth, variable=\x] ({-\x},{sqrt(0.5)-0.5-sqrt(\x)}) -- (0.5,-0.5) -- plot [domain=0.5:2, smooth, variable=\x] ({\x},{sqrt(0.5)-0.5-sqrt(\x)}) -- cycle;

\node at (0,-0.1) {$0$};

\node [fill=white, draw=black] at (0, 0.5) {Proposition \ref{maplemma}};
\node [fill=white, draw=black] at (0, -0.7) {Proposition \ref{2t6}};

\node [label=right:$|\operatorname{Re}{\lambda}|\sim |\operatorname{Im}\lambda|^{p-}$] at (0.7, -0.6) {};
\node [circle, fill,inner sep=0.4mm, outer sep=0pt, red] at (0.2, -0.85) {};
\draw [->, red, line width=0.25mm] (0.2, -0.85) -- (0.5,-1.2);
\node [circle, fill,inner sep=0.4mm, outer sep=0pt, red] at (-0.2, -0.85) {};
\draw [->, red, line width=0.25mm] (-0.2, -0.85) -- (-0.5,-1.2);
\end{tikzpicture}}
}\hfill
\subfloat[$L^\infty$-damping.]{
\includegraphics[page=4,width=0.48\linewidth]{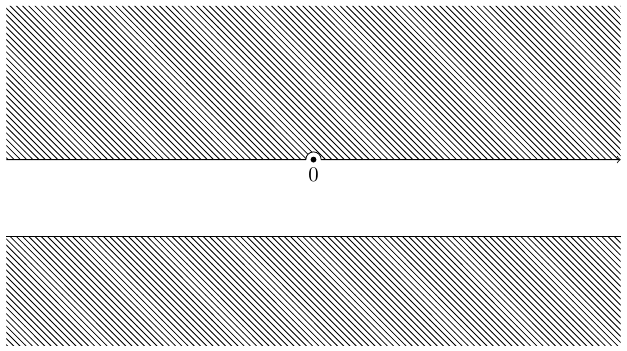}
}
\caption{Pole-free regions of $P_\lambda$ for normally $L^p$-damping in $\lambda\in\mathbb{C}$.}
\label{f3}
\end{figure}

We now do some spectral analysis of $P_\lambda=-\Delta-i\lambda W-\lambda^2$. We say $P_\lambda$ has a pole at $\lambda\in\mathbb{C}$ if $P_\lambda: H^1\rightarrow H^{-1}$ fails to be invertible. We show in Proposition \ref{maplemma}, when $p\in [1,\infty]$, we have no poles in the upper half plane. We show in Proposition \ref{2t6}, when $p\in (1,\infty)$, we have no poles in some regions in the lower half plane that shrink as $p\rightarrow1^+$. The pole-free region of normally $L^p$-damping in the lower half plane is asymptotically smaller than that of $L^\infty$-damping: see Figure \ref{f3}. 

We begin with the upper half plane:
\begin{proposition}[Pole-free region in the upper half plane]\label{maplemma}
Let $W$ be normally $L^1$. For $\lambda\in \mathbb{C}$, consider $\Pl=-\Delta-i\lambda W-\lambda^2$ as a bounded operator from $H^1$ to $H^{-1}$. 
Then the following are true:
\begin{enumerate}[wide]
\item $\Pl: H^{1} \rightarrow H^{-1}$ is bijective on $\lambda\in \{\cim \lambda\ge 0, \lambda\neq 0\}$. 
\item There is $C>0$ such that for all $\lambda\in \{\cim\lambda>0\}$, we have
\begin{equation}\label{2l20}
\|P_\lambda^{-1}\|_{L^2\rightarrow H^1}\le \frac{C}{\cim\lambda}, \ \|P_\lambda^{-1}\|_{L^2\rightarrow L^2}\le \frac{C}{\abs{\cim\lambda}^2}.
\end{equation}
\item There is $C>0$ such that for for any $u\in H^1$ and any $\lambda\in \{\cim\lambda\ge0\}$, we have
\begin{equation}
\|u\|_{H^1}^2\le C\langle \cre\lambda\rangle^2\|u\|_{L^2}^2+C\|P_\lambda u\|_{H^{-1}}^2.
\end{equation} 
\item For any $\lambda\in\mathbb{C}$, $P_\lambda:H^1\rightarrow H^{-1}$ is bijective if and only if $P_{-\bar\lambda}:H^1\rightarrow H^{-1}$ is bijective: $\spec P_\lambda$ is symmetric about the imaginary axis. 
\item At $\lambda=0$, $P_0=-\Delta$ has a simple pole. $P_0$ is surjective from $H^1$ to $\{f\in H^{-1}: \langle f,1\rangle=0\}$, and $\ker P_0=\spa \{1\}$. 
\end{enumerate}
\end{proposition}
\begin{proof}
1. First we will show that $\Pl$ is Fredholm with index $0$ for all $\lambda\in \mathbb{C}$. Note
\begin{equation}
\left\langle P_{i}u,v \right\rangle=\left\langle \left(-\Delta+1+W\right)u, v\right\rangle,
\end{equation}
is a coercive form on $H^1$. Indeed we have
\begin{equation}
\abs{\left\langle P_{i}u,v \right\rangle}\le\abs{\langle \nabla u, \nabla v\rangle+\langle u,v\rangle+\langle \sqrt{W}u,\sqrt{W} v\rangle}\le C\|u\|_{H^1}\|v\|_{H^1}. 
\end{equation}
and
\begin{equation}
\left\langle P_{i}u,u \right\rangle=\|u\|_{H^1}^2+\|\sqrt{W}u\|^2\ge \|u\|_{H^1}^2.
\end{equation}
By the Lax-Milgram theorem, we have
\begin{equation}
P^{-1}_{i}: H^{-1}\rightarrow H^1,
\end{equation}
is bounded. Then $P_{i}$ is Fredholm with index $0$. Now note
\begin{equation}
\Pl=P_{i}-\left(1+\lambda^2\right)-\left(1+i\lambda\right)W,
\end{equation}
and that $\left(1+\lambda^2\right)+\left(1+i\lambda\right)W: H^1 \rightarrow H^{-\frac{1}{2}-}\hookrightarrow H^{-1}$ compactly from Lemma \ref{2t3}. Thus $\Pl$ is Fredholm with index 0 for all $\lambda \in \mathbb{C}$. This also implies that $P_\lambda$ is bijective if and only if $P^*_\lambda=P_{-\bar\lambda}$ is bijective from $H^1$ to $H^{-1}$. 

2. We show that $\Pl^*: H^1 \rightarrow H^{-1}$ has trivial kernel in the upper half plane. For $\cim\lambda \geq 0, \lambda \neq 0$ let $\lambda=\alpha+i\beta$. Consider
\begin{equation}
\Pl^* u=P_{-\bar{\lambda}} u=\left(-\Delta +\left(\beta^2+\beta W-\alpha^2\right)+i\alpha(W+2\beta)\right)u =0.
\end{equation}
Pair it with $u$ to see
\begin{equation}\label{eq:kernel}
\left\langle \Pl^* u, u\right\rangle=\|\nabla u\|^2+\left(\beta^2-\alpha^2\right)\|u\|^2+\beta\|\sqrt{W}u\|^2+i\left(\alpha\|\sqrt{W}u\|^2+2\alpha\beta\|u\|^2\right) =0. 
\end{equation}
Suppose $\beta>0$. When $\alpha\neq 0$, the imaginary part of \eqref{eq:kernel} implies $\|u\|=0$. When $\alpha=0$, the real part of \eqref{eq:kernel} implies $\|u\|=0$. When $\beta=0$ and $\alpha\neq 0$, \eqref{eq:kernel} is reduced to 
\begin{equation}
\|\nabla u\|^2-\alpha^2\|u\|^2-i\alpha\|\sqrt{W}u\|^2=0,
\end{equation}
and $\|\sqrt{W}u\|=0$. From the unique continuation we know $u\equiv 0$ almost everywhere. Now since $\ker{P_\lambda^*}$ is trivial and $P_\lambda$ has index $0$, we know $\operatorname{CoKer}P_\lambda$ is trivial and $P_\lambda$ is invertible. 

3. Let $\lambda=\alpha+i\beta$ with $\beta>0$. Consider 
\begin{equation}
P_{\lambda} u=(-\Delta+(\beta^2-\alpha^2+\beta W)-i\alpha(W+2\beta))u=f. 
\end{equation}
Pair it with $u$ to observe
\begin{equation}\label{2l21}
\langle P_\lambda u, u\rangle=\|\nabla u\|^2+(\beta^2-\alpha^2)\|u\|^2+\beta\|\sqrt{W}u\|^2-i\alpha(2\beta\|u\|^2+\|\sqrt{W}u\|^2). 
\end{equation}
When $\alpha=0$, we have
\begin{equation}
\|\nabla u\|^2+\beta^2\|u\|^2\le \frac{1}{2}\beta^{-2}\|f\|+\frac{1}{2} \beta^2\|u\|^2,
\end{equation}
the absorption of the last term gives \eqref{2l20}. When $\alpha\neq 0$, take the imaginary part to see
\begin{equation}
2\abs{\alpha}\beta\|u\|^2+\abs{\alpha}\|\sqrt{W}u\|^2\le \abs{\alpha}\beta\|u\|^2+C\beta^{-1}\abs{\alpha}^{-1}\|f\|^2,
\end{equation}
the absorption of the first term on the right gives
\begin{equation}\label{2l22}
\|u\|^2\le C\beta^{-2}\abs{\alpha}^{-2}\|f\|^2.
\end{equation}
The real part of \eqref{2l21} reads
\begin{equation}
\|\nabla u\|^2+\beta^2\|u\|^2+\beta\|\sqrt{W}u\|^2\le C\epsilon^{-1}\beta^{-2}\|f\|^2+\alpha^2\|u\|^2+\epsilon \beta^2\|u\|^2.
\end{equation}
Absorb the last term on the right and bring in \eqref{2l22} to see
\begin{equation}
\|\nabla u\|^2+\beta^2\|u\|^2\le C\beta^{-2}\|f\|^2,
\end{equation}
which is \eqref{2l20}. 

4. Since $\beta\ge0$, the real part of \eqref{2l21} gives
\begin{equation}
\|\nabla u\|^2+(\beta^2-\alpha^2)\|u\|^2\le C\epsilon^{-1}\|f\|_{H^{-1}}^2+\epsilon \|u\|_{H^1}^2,
\end{equation}
and the absorption of the last term gives
\begin{equation}
\|u\|_{H^1}^2\le C\langle \alpha\rangle^2\|u\|_{L^2}^2+C\|f\|_{H^{-1}}^2,
\end{equation}
where $C$ does not depend on $\beta$. 

5. Let $\lambda=0$. Then $\ker P_0^*=\ker P_0=\ker(-\Delta)=\spa \{1\}$, and $P_0$ is surjective from $H^1$ to $(\ker P_0^*)^{\perp}=\{f\in H^{-1}:\langle f, 1\rangle=0\}$.
\end{proof}

We now look at the lower half plane:

\begin{lemma}[Interpolation inequality]\label{2t5}
Let $W$ be normally $L^p$ for $p\in[1,\infty)$, then for all $r\in (\frac{1}{2p},1)$ there exists $C>0$ such that for all $\gamma>0$,
\begin{equation}
\|\sqrt{W}u\|_{L^2}\le \gamma\|u\|_{L^2}+C\gamma^{\frac{1-r}{-r}}\|u\|_{H^1}. 
\end{equation}
Note that $C$ does not depend on $\gamma, u$. 
\end{lemma}
\begin{proof}
Use the weight $(C^{-1}\gamma)^{\frac{s_2-r}{s_1-r}}$ in \cite[Proposition E.21]{dz19} with $s_1=0, s_2=1, h=1$ to observe that for all $r\in (0,1)$, there is $C>0$ such that for all $\gamma>0$ we have
\begin{equation}\label{2l31}
\|u\|_{H^r}\le \gamma\|u\|_{L^2}+C\gamma^{\frac{1-r}{-r}}\|u\|_{H^1}. 
\end{equation}
For any $r\in(\frac{1}{2p},1)$, Lemma \ref{2t3} implies $\|\sqrt{W} u\|_{L^2}\le \|u\|_{H^r}$ and concludes the proof. 
\end{proof}
\begin{proposition}[Pole-free region in the lower half plane]\label{2t6}
Let $W$ be normally $L^p$ for $p\in (1,\infty)$. Then the following are true:
\begin{enumerate}[wide]
\item For any $M, \delta>0$, there exists $K_{M,\delta}>0$ such that $\Pl: H^{1} \rightarrow H^{-1}$ is bijective on $\{\lambda\in \mathbb{C}: \abs{\cre{\lambda}}\le -M \abs{\cim\lambda}^{p-\delta}, \cim\lambda\le -K_{M,\delta}\}$. 
\item Along any ray $\lambda=\abs{\lambda} e^{i\theta}$ with $\theta\in(\pi, 2\pi)$, there are $\lambda_0, C>0$ such that for any $\abs{\lambda}> \lambda_0$ we have
\begin{equation}\label{2l32}
\|P_{\lambda}^{-1}\|_{L^2\rightarrow H^1}\le C\abs{\lambda}^{-1}, \|P_{\lambda}^{-1}\|_{L^2\rightarrow L^2}\le C\abs{\lambda}^{-2}. 
\end{equation}
\item There are $C,K>0$ such that for any $u\in H^1$ and $\lambda\in \{\cim\lambda\le -K\}$ we have
\begin{equation}
\|u\|_{H^1}^2\le C\langle \cre\lambda\rangle^2\|u\|_{L^2}^2+C\|P_\lambda u\|_{H^{-1}}^2.
\end{equation} 
\end{enumerate}
\end{proposition}
\begin{proof}
1. Let $\lambda=\alpha+i\beta$. We show $P^*_\lambda$ has trivial kernel on 
\begin{equation}
D_\delta=\{0<\abs{\alpha}\le -M\beta^{1+\delta},\ \beta\le -K_\delta\}. 
\end{equation}
Let $P_\lambda^* u=0$. Fix some $r\in (\frac{1}{2p-2\delta}, \frac{1}{2})\subset(\frac{1}{2p},\frac{1}{2})$, then $s=\frac{1-r}{r}>2p-2\delta-1$. Lemma \ref{2t5} implies 
\begin{equation}\label{2l24}
\|\sqrt{W}u\|_{L^2}\le \epsilon\abs{\beta}^{\frac{1}{2}}\|u\|_{L^2}+C_{r,\epsilon}\abs{\beta}^{-\frac{s}{2}}\|u\|_{H^1}. 
\end{equation}

1a. When $\alpha\neq 0$, the imaginary part of \eqref{eq:kernel} reads
\begin{equation}\label{2l15}
2\abs{\beta}\|u\|^2\le \|\sqrt{W}u\|^2\le \abs{\beta}\|u\|^2_{L^2}+C_r\abs{\beta}^{-s}\|u\|_{H^1}^2,
\end{equation}
which implies that for $\abs{\beta}\ge C_r$ we have
\begin{equation}
\|u\|^2\le C_r\abs{\beta}^{-s-1}\|\nabla u\|^2,
\end{equation}
since $-s<-1$. Substituting this back into \eqref{2l15} to observe
\begin{equation}
\|\sqrt{W}u\|^2\le C_r\abs{\beta}^{-s}\|\nabla u\|^2.
\end{equation}
 The real part of \eqref{eq:kernel} now reads
\begin{equation}
\|\nabla u\|^2+\beta^2\|u\|^2=\alpha^2\|u\|^2+\abs{\beta}\|\sqrt{W}u\|^2\le C_r(\alpha^2\abs{\beta}^{-s-1}+\abs{\beta}^{-s+1})\|\nabla u\|^2.
\end{equation}
Note that $\alpha\le M\abs{\beta}^{p-\delta}$, and we have
\begin{equation}
\|\nabla u\|^2+\beta^2\|u\|^2\le C_r\abs{\beta}^{-s+2p-2\delta-1}\|\nabla u\|^2.
\end{equation}
Since $s>2p-2\delta-1$, there exists $K_\delta>0$ large such that for $\abs{\beta}\ge K_\delta$, the right hand side of the last equation can be absorbed by the left. We thus obtained $\|u\|_{H^1}=0$.

1b. When $\alpha=0$, consider the real part of \eqref{eq:kernel}:
\begin{equation}
\|\nabla u\|^2+\beta^2\|u\|^2=\abs{\beta}\|\sqrt{W}u\|^2\le \epsilon^2\beta^2\|u\|^2+C_{r,\epsilon}\abs{\beta}^{-s+1}\|u\|_{H^1}. 
\end{equation}
Since $s>1$, both terms on the right can be absorbed by the left when $\abs{\beta}$ is sufficiently large. In both cases, $u=0$ in $H^1$. Thus $P^*_\lambda$ has trivial kernel in $D_\delta$ and $P_\lambda: H^{1}\rightarrow H^{-1}$ is bijective there. 

2. Fix $\theta\in (\pi,2\pi)$. Then $\lambda=\abs{\lambda}e^{i\theta}$ can be parametrized by $\lambda= h^{-1}(\delta-i)$, where $h=\langle \delta\rangle\abs{\lambda}^{-1}\rightarrow 0$ and $\delta=\cot\theta$ fixed. From Step 1 we know that $P_\lambda$ is bijective for $\abs{\lambda}$ large. Let $P_\lambda u=f$ where
\begin{equation}
P_\lambda=-\Delta+h^{-2}(1-\delta^2)-h^{-1} W+ih^{-1}(2h^{-1}-\delta W).
\end{equation}
Semiclassicalize by $P_h=h^2P_\lambda$, $g=h^2f$ and
\begin{equation}
P_h u=-h^2\Delta+(1-\delta^2)-h W+i(2-h \delta W)u=g. 
\end{equation}
Pair it with $u$ in $L^2$ to observe
\begin{equation}\label{2l16}
\|h\nabla u\|^2+(1-\delta^2)\|u\|^2-h\|\sqrt{W}u\|^2+i\left(2\|u\|^2-h\delta\|\sqrt{W}u\|^2\right)=\langle g,u\rangle. 
\end{equation}
Fix some $r\in (\frac{1}{2p},\frac{1}{2})$, then $s=\frac{1-r}{r}>1$. Lemma \ref{2t5} implies 
\begin{equation}\label{2l18}
\|\sqrt{W}u\|_{L^2}\le \epsilon \langle\delta\rangle^{-\frac{1}{2}}h^{-\frac{1}{2}}\|u\|_{L^2}+C_\epsilon \langle\delta\rangle^{\frac{s}{2}} h^{\frac{s}{2}}\|u\|_{H^1}. 
\end{equation}
The imaginary part of \eqref{2l16} implies
\begin{multline}
2\|u\|^2\le h\delta\|\sqrt{W}u\|^2+\epsilon\|u\|^2+C\epsilon^{-1}\|g\|^2\le \epsilon(1+\delta\langle \delta\rangle^{-1})\|u\|^2_{L^2}+C_\epsilon (\delta\langle \delta\rangle^{s}h^{s+1}\|u\|^2_{H^1}+\|g\|^2)\\
\le 2\epsilon\|u\|^2_{L^2}+C_\epsilon (\delta\langle \delta\rangle^{s}h^{s+1}\|u\|^2_{H^1}+\|g\|^2).
\end{multline}
When $h$ is smaller than some $\delta$-dependent bounds, the absorption of the first two terms on the right gives
\begin{equation}
\|u\|^2\le C h^{s-1}\|h\nabla u\|^2+C\|g\|^2.
\end{equation}
From now on our constants also depend on $\delta$. The imaginary part of \eqref{2l16} then gives
\begin{equation}\label{2l33}
h\|\sqrt{W}u\|_{L^2}^2\le C h^{s-1}\|h\nabla u\|^2+C\|g\|^2.
\end{equation}
Substitute those two estimates back to the real part of \eqref{2l16} to see
\begin{equation}
\|h\nabla u\|^2+\|u\|^2=\cre \langle g,u\rangle + h\|\sqrt{W}u\|^2+\delta^2\|u\|^2\le Ch^{s-1}\|h\nabla u\|^2+C\|g\|^2.
\end{equation}
Note $s>1$ and for $h$ small, $Ch^{s-1}<1$. Thus by absorbing the first term on the right there is $h_0>0$ and for all $h<h_0$ we have $\|u\|_{H^1_h}\le C\|g\|_{L^2}$, where $C$ does not depend on $h$. This implies
\begin{equation}\label{2l17}
\|u\|_{L^2}^2+\langle \delta\rangle^2\abs{\lambda}^{-2}\|\nabla u\|_{L^2}^2\le C\langle \delta\rangle^{4}\abs{\lambda}^{-4}\|f\|_{L^2}^2,
\end{equation}
for all $\abs{\lambda}\ge \lambda_0$ for some $\lambda_0>0$. This gives us the desired estimate \eqref{2l32} on the lines.

3. Consider the pairing $\langle P_\lambda u, u\rangle$ in \eqref{2l21}. Its real part is
\begin{equation}
\|\nabla u\|^2+(\beta^2-\alpha^2)\|u\|^2=\cre\langle f,u\rangle-\beta\|\sqrt{W}u\|^2.
\end{equation}
Apply \eqref{2l24} to observe
\begin{equation}
\|\nabla u\|^2+(\beta^2-\alpha^2)\|u\|^2\le C\epsilon^{-1}\|f\|_{H^{-1}}^2+\epsilon \|u\|_{H^1}^2+\epsilon\abs{\beta}^{2}\|u\|^2+C_{\epsilon}\abs{\beta}^{-s+1}\|u\|^2_{H^1}.
\end{equation}
As $s>1$, there exists $K>0$ such that when $\beta\le -K$, the last three terms on the right can be absorbed. This gives
\begin{equation}
\|u\|_{H^1}^2-C\langle \alpha\rangle^2\|u\|^2\le C\|f\|_{H^{-1}}^2,
\end{equation}
as desired. 
\end{proof}
\begin{remark}
The proof only used the property of $\sqrt{W}$ being a bounded map from $H^{r}$ to $L^2$ for some $r<\frac{1}{2}$. Note that we need some elbow room for $r\in(\frac{1}{2p},\frac{1}{2})$ to show $h^{\frac{1}{2}}\|\sqrt{W}u\|\le Ch^{\frac{1}{2}-r}\|u\|_{H^r_h}$ is semiclassically small compared to $\|u\|_{H^1_h}$ in \eqref{2l33}. When $p=1$, $(\frac{1}{2p},\frac{1}{2})$ is empty and there is no $r$ for the proof to work. This is consistent with the observation that $\sqrt{W}$ is not a bounded map from $H^{\frac{1}{2}}$ to $L^2$ when $p=1$. In other words,  $h^{\frac{1}{2}}\sqrt{W}u$ becomes too large compared to all the other terms in \eqref{2l16}. This may explain why there can be finite time extinction in the case $\beta=-1$ in \cite{fhs20,cc01}. 
\end{remark}

\subsection{Semigroup generated by normally \texorpdfstring{$L^p$}{Lp}-damping}
Now we consider semigroups for normally $L^1$-damping. 
\begin{definition}[Semigroup for normally $L^1$-damping]
Let $\Hc=\{(u,v) \in H^1(M) \times L^2(M)\}$ with norm $\|(u,v)\|_{\Hc} = \|u\|_{H^1}^2 + \ltwo{v}^2$. Define 
\begin{equation}
\Ac = \begin{pmatrix} 0 & \id \\ \Delta & -W \end{pmatrix}: D(\Ac) \ra \Hc,
\end{equation}
with $D(\Ac)=\Dc=\{(u,v)\in \mathcal{H}: \Ac(u,v)\in \mathcal{H}\}$. Note that the equivalent definition for $\Dc$ is
\begin{equation}
\Dc=\{(u,v)\in H^1\times H^1: -\Delta u+Wv\in L^2\},
\end{equation}
equipped with $\|(u,v)\|_\Dc^2=\|(u,v)\|_{H^1\times H^1}^2+\|-\Delta u+Wv\|_{L^2}^2$, equivalent to the graph norm of $\Ac$. Solutions of the damped wave equation\eqref{DWE} are equivalent to solutions of 
\begin{equation}
\begin{cases} \p_t U(t) = \Ac U(t) \\ U(0)= (u_0, u_1)^t \end{cases}, \quad U(t) = \begin{pmatrix} u(t,x) \\ \p_t u(t,x) \end{pmatrix},
\end{equation}
where $U(t)=e^{t\Ac}U(0)$, whose semigroup nature will be shown soon in Propositions \ref{2t18} and \ref{2t12}. 
\end{definition}

\begin{remark}
Note that the our semigroup construction applies to damping not covered by \cite{fst18}. In particular, damping functions in $X^{\beta}_{\theta}$ are normally $L^1$ when $-1<\beta$, but are not $L^2_{\text{loc}}$ for $-1<\beta<-\frac{1}{2}$, and cannot satisfy the relative boundedness piece of \cite[Assumption 1]{fst18}. The semigroup from \cite{fst18} is used in \cite{fhs20} on damping of the form $2/x$ on $(0,1)$ with Dirichlet boundary conditions. This is possible because solutions are $0$ exactly where the damping is singular, which allows use of the Hardy inequality. We allow solutions to be non-zero where the damping is singular, and indeed, this is an essential feature of our quasimode construction in Section \ref{sharpnesssection}. One can check that our semigroup is the same as in \cite[Section 2.2]{cpsst19}.
\end{remark}

We remind the reader that $P_\lambda=-\Delta-i\lambda W-\lambda^2$ is bounded from $H^1\rightarrow H^{-1}$. We now draw the connection between $P_\lambda$ and $\Ac+i\lambda$, which will be useful to show that $e^{t\Ac}$ is a strongly continuos semigroup. Let $k\ge 2$ and $\Dc^1=\Dc$, define $\Dc^k=\{(u,v)\in \Dc^{k-1}: \Ac(u,v)\in \Dc^{k-1}\}$, and then $\Ac^k: \Dc^k\rightarrow\mathcal{H}$ is a bounded map. 
\begin{lemma}[Spectral equivalence]\label{2t9}
Let $\lambda\in \mathbb{C}$. Then the following are true:
\begin{enumerate}
\item $\Ac+i\lambda:\Dc\rightarrow\mathcal{H}$ is bijective if and only if $P_{\lambda}: H^{1}\rightarrow H^{-1}$ is bijective. 
\item $\Ac+i\lambda$ is bijective iff $\Ac-i\bar\lambda$ is so: $\spec \Ac$ is symmetric about the real axis. 
\item $\Ac: \Dc\rightarrow\{(u,v)\in \mathcal{H}: \langle Wu+v,1\rangle =0\}$ is surjective, and $\ker \Ac=\spa \{(1,0)\}$. 
\item If $W\neq 0$, then $\ker(\Ac^k)=\ker(\Ac)$ for all $k\ge 2$: $\Ac$ has a simple pole at $0$. 
\end{enumerate}
\end{lemma}
\begin{proof}
1. Assume $P_\lambda:H^1\rightarrow H^{-1}$ is injective at some $\lambda$. Then for any $(f,g)\in \mathcal{H}=H^1\times L^2$, consider
\begin{equation}\label{2l25}
\begin{pmatrix}
u\\
v
\end{pmatrix}=
(\Ac+i\lambda)^{-1}
\begin{pmatrix}
f\\g
\end{pmatrix}
=
\begin{pmatrix}
-i\lambda^{-1}(\id +P_\lambda^{-1}\Delta) & -P_\lambda^{-1}\\
-P_\lambda^{-1}\Delta & i\lambda P_\lambda^{-1}
\end{pmatrix}
\begin{pmatrix}
f\\g
\end{pmatrix}.
\end{equation}
as an element in $H^1\times H^1$ that satisfies
\begin{equation}
(\Ac+i\lambda)
\begin{pmatrix}
u\\v
\end{pmatrix}=
\begin{pmatrix}
i\lambda u+v\\
\Delta u-\left(W-i\lambda\right)v
\end{pmatrix}=
\begin{pmatrix}
f\\g
\end{pmatrix}.
\end{equation}
Note that $\Delta u-Wv=g+i\lambda v\in L^2$. Thus $(u,v)\in \Dc$ and $(\Ac+i\lambda)(u,v)=(f,g)$. Moreover, $(\Ac+i\lambda)^{-1}(0,0)=(0,0)$ since $P_{\lambda}$ is injective. Therefore $\Ac+i\lambda$ is bijective. 

2. Assume $\Ac+i\lambda:\Dc\rightarrow \mathcal{H}$ is bijective. For any $g\in L^2$, there exists a unique $(u,v)\in \Dc$ such that
\begin{equation}
(\Ac+i\lambda)
\begin{pmatrix}
u\\v
\end{pmatrix}=
\begin{pmatrix}
i\lambda u+v\\
\Delta u-\left(W-i\lambda\right)v
\end{pmatrix}=
\begin{pmatrix}
0\\-g
\end{pmatrix}.
\end{equation}
Thus $v=-i\lambda u$ and $P_\lambda u=g$. Moreover, $(-u, -i\lambda u)\in \Dc$ implies $u\in \{u\in H^1: (-\Delta-i\lambda W)u\in L^2\}\subset H^1$. Thus $P_\lambda: L^2 \rightarrow \{u\in H^1: (-\Delta-i\lambda W)u\in L^2\}$ is bijective. Its dual $P_\lambda^*=P_{-\bar\lambda}: \{u\in H^1: (-\Delta-i\lambda W)u\in L^2\}^*\rightarrow L^2$ is bijective. Let $f\in H^{-1}\subset (-\Delta-i\lambda W)u\in L^2\}^*$, then there exists a unique $w\in L^2$ such that $P_{-\bar\lambda} w= (-\Delta+i\bar \lambda W-\bar\lambda^2)w=f$. Pair $P_{-\bar\lambda} w$ with $w$ and take the real part to see
\begin{equation}
\|\nabla w\|^2\le \abs{\langle f, w\rangle}+\bar\lambda^2 \|w\|^2<\infty.
\end{equation} 
Thus $w\in H^1$ and $P_{-\bar\lambda}: H^{-1}\rightarrow H^1$ is bijective. Apply Proposition \ref{maplemma} to see $P_{\lambda}: H^{-1}\rightarrow H^1$ is also bijective. This also implies $\Ac+i\lambda$ is bijective iff $\Ac-i\bar\lambda$ is so. 

3. Let $(u,v)\in \mathcal{H}$ and $\Ac(u,v)=(v,\Delta u-Wv)=(f,g)\in\mathcal{H}$. Thus $v=f\in H^1$ and $-P_0 u=\Delta u=Wf+g\in H^{-1}$ since $Wf\in H^{-1}$. Due to Proposition \ref{maplemma}, there exists $u\in H^1$ if and only if $\langle Wf+g,1\rangle=0$, and $\ker \Delta=\spa \{1\}$. Thus $\Ac$ is surjective onto $\{(f,g)\in\mathcal{H}: \langle Wf+g,1\rangle=0\}$ and $\ker \Ac=\spa \{(1,0)\}$. 

4. To show $\ker(\Ac^k)=\ker(\Ac)$, it suffices to show that $\ker(\Ac^2)=\ker(\Ac)$, that is, for any $(u,v)\in \Dc^2$, if $\Ac(u,v)\in\ker\Ac$ then $(u,v)\in \ker\Ac$. Note that $\Dc^2=\{(u,v)\in H^1\times H^1: \Delta u-Wv\in H^1, -\Delta v+W(\Delta u-Wv)\in L^2\}$. Assume $\Ac(u,v)=(v,\Delta u-Wv)=c(1,0)$ for some constant $c$. Then $v=c$, and $\Delta u=cW\in H^{-1}$. By Proposition \ref{maplemma}, there exists $u\in H^1$ such that $\Delta u=cW\in H^{-1}$ only if $\langle cW, 1\rangle=0$. Since $W\ge 0$ is not identically 0, $v=c=0$, and $u=c'$ for some constant $c'$. Thus $(u,v)\in \ker \Ac$. Note that when $W\equiv 0$, $\ker(\Ac^2)=\{(u,v)\in \Dc^2: \langle u,1\rangle=\langle v,1\rangle=0\}$ is a 2-dimensional subspace that contains $\ker\Ac$ as a proper subspace. 
\end{proof}

\begin{corollary}[Pole-free region of $\Ac$]\label{2t10}
The following are true:
\begin{enumerate}[wide]
\item Let $W$ be normally $L^1$. Then $\Ac-\mu:\Dc\rightarrow \mathcal{H}$ is bijective on $\{\cre\mu\ge0, \mu\neq 0\}$. 
\item If we further assume $W$ is normally $L^p$ for $p\in (1,\infty]$, then for any $M, \delta>0$, there exists $K_{M,\delta}>0$ such that $\Ac-\mu:\Dc\rightarrow \mathcal{H}$ is bijective on $\{\mu\in \mathbb{C}: \abs{\cim{\mu}}\le -M \abs{\cre\mu}^{p-\delta}, \cre\mu\le -K_{M,\delta}\}$.
\end{enumerate}
\end{corollary}
\begin{proof}
Let $\lambda=i\mu$ and apply Lemma \ref{2t9} to the part 1 of Propositions \ref{maplemma} and \ref{2t6}. 
\end{proof}

Let $w_0=\essinf\{W(z): z\in M\}\ge 0$ be the essential minimum of the damping $W$. We now show $e^{t\Ac}$ is a strongly continuous semigroup on $\Hc$. 
\begin{proposition}[Quasi-contraction semigroup]\label{2t18}
Let $W$ be normally $L^1$. Then the following are true:
\begin{enumerate}
\item The generator $\Ac$ is closed and $\Dc$ is dense in $\mathcal{H}$. 
\item The generator $\Ac$ generates a strongly continuous semigroup $e^{t\Ac}: \Hc\rightarrow\Hc$. 
\item $e^{t\Ac}$ is quasi-contractive: for all $t\ge 0$, $\|e^{t\Ac}\|_{\mathcal{L}(\mathcal{H})}\le \exp(\frac{1}{2}(\langle w_0\rangle -w_0) t)$. 
\item The generator $\Ac$ has compact resolvent, and the spectrum of $\Ac$ contains only isolated eigenvalues.
\end{enumerate}
\end{proposition}
\begin{proof}
1. Consider the core $\Dc_0=C^\infty(M)\times \{v\in C^\infty(M): \supp v\subset M\setminus N\}\subset \Dc$. Indeed, for such $(u,v)\in \Dc_0$, we have $Wv\in L^2$ from Lemma \ref{2t4}, and thus $(u,v)\in \Dc$. Furthermore $\Dc_0$ is dense in $\mathcal{H}=H^1\times L^2$. 

2. Consider the Hilbert space $H^1_W=\{\phi\in H^1: W \phi\in{L^2}\}$ equipped with inner product 
\begin{equation}
\langle\phi,\psi\rangle_{H^1_W}=\langle \phi,\psi\rangle_{H^1} +\langle W\phi, W\psi\rangle. 
\end{equation}
Consider its dual space $H^{-1}_{W}=(H^{1}_W)^*$, the set of complex-valued continuous linear functionals on $H^1_W$. Note that $H^{-1}\subset H^{-1}_W$. The map $G:\mathcal{H}\times \mathcal{H}\rightarrow L^2\times H^{-1}_W$ given by 
\begin{equation}
G\left((u,v), (f,g)\right)=(\Ac (u,v)-(f,g))=(v-f,\Delta u-W v-g),
\end{equation}
is then bounded. The graph of $\Ac: \Dc\rightarrow\mathcal{H}$ is the zero set of this continuous map $G$, and is thus closed. Indeed, $\Ac:\Dc\rightarrow \mathcal{H}$ is the maximal closed extension of $\Ac:\Dc_0\rightarrow C^\infty\times {L^2}$.

3. We now show $\Ac$ generates a strongly continuous semigroup. Firstly Consider that for any $(u,v)\in \Dc$, we have $v\in H^1$ and $\sqrt{W}v\in L^2$, and
\begin{equation}
\langle \Ac(u,v), (u,v)\rangle_{\mathcal{H}}=\langle (v,\Delta u-Wv),(u,v)\rangle_{\mathcal{H}}=-\|\sqrt{W} v\|^2+\langle v, u\rangle+2i\cim\langle \nabla v,\nabla u\rangle,
\end{equation}
the real part of which is
\begin{equation}
\cre \langle \Ac(u,v), (u,v)\rangle_{\mathcal{H}}=-\|\sqrt{W} v\|^2+\cre\langle v, u\rangle\le -w_0\|v\|^2+\cre\langle v, u\rangle.
\end{equation}
Note $\frac12 (\sqrt{w_0^2+1}-w_0)^{-1}-w_0=\frac12 (\sqrt{w_0^2+1}-w_0)$ and 
\begin{equation}
\abs{\cre\langle v, u\rangle}\le \frac12 (\sqrt{w_0^2+1}-w_0)^{-1}\|v\|^2 +\frac12 (\sqrt{w_0^2+1}-w_0)\|u\|^2,
\end{equation}
implies $\cre \langle \Ac(u,v), (u,v)\rangle_{\mathcal{H}}\le \frac12 (\langle w_0\rangle -w_0)\|(u,v)\|_{\mathcal{H}}^2$. Now for any real $\mu>\frac12 (\langle w_0\rangle -w_0)>0$, Corollary \ref{2t10} implies $\Ac-\mu: \Dc\rightarrow\mathcal{H}$ is bijective. Moreover, 
\begin{equation}
\|(\Ac-\mu)(u,v)\|_\mathcal{H}\|(u,v)\|_{\mathcal{H}}\ge \cre \langle (\mu-\Ac)(u,v), (u,v)\rangle_{\mathcal{H}}\\
\ge (\mu-\frac12 (\langle w_0\rangle -w_0))\|(u,v)\|_{\mathcal{H}}^2.
\end{equation}
This implies $\|(\Ac-\mu)^{-1}\|_{\mathcal{L}(\mathcal{H})}\le \left(\mu-\frac12 (\langle w_0\rangle -w_0)\right)^{-1}$. Apply the Hille-Yoshida theorem as in \cite[Corollary II.3.6, p. 76]{en99} to conclude that $\Ac: \Dc\rightarrow\mathcal{H}$ generates a strongly continuous semigroup $e^{t\Ac}:\mathcal{H}\rightarrow \mathcal{H}$, and $\|e^{t\Ac}\|_{\mathcal{L}(\mathcal{H})}\le \exp(\frac{1}{2}(\langle w_0\rangle -w_0) t)$ for all $t\ge 0$. 

4. Note that the range of any resolvent of $\Ac$ will be a subset of $\Dc$. For any $(u,v)\in \Dc$, we have $Wu\in H^{-\frac{1}{2}-}$ from Lemma \ref{2t3}, and $-\Delta u\in H^{-\frac{1}{2}-}$. From the classical elliptic regularity, we have $u\in H^{\frac{3}{2}-}$. Thus $\Dc\subset H^{\frac{3}{2}-}\times H^1$ embeds compactly into $\mathcal{H}=H^1\times L^2$. Then any resolvent of $\Ac$ is a compact operator from $\mathcal{H}$ to $\mathcal{H}$. This implies that the spectrum of $\Ac$ contains only isolated eigenvalues. 
\end{proof}
\begin{remark}
Note that when $W\equiv 0$, the $L^2$-mass of $u$ may grow linearly in time. We later show that if $W\neq 0$ then $e^{t\Ac}$ is bounded in Proposition \ref{2t12}.
\end{remark}

We need a more quantitative connection between the resolvent estimates for $P_\lambda$ and $\Ac+i\lambda$ to prove further results about $e^{t\Ac}$: 
\begin{lemma}[Resolvent equivalence]\label{2t8}
The following are true:
\begin{enumerate}[wide]
\item At any $\lambda\in \mathbb{C}\setminus \{0\}$ such that $\Ac+i\lambda: \Dc\rightarrow\mathcal{H}$ is bijective, we have
\begin{equation}
\|P_\lambda^{-1}\|_{L^2\rightarrow H^1}\le \|(\Ac+i\lambda)^{-1}\|_{\mathcal{L}(\mathcal{H})}, \ \|P_\lambda^{-1}\|_{L^2\rightarrow L^2}\le \abs{\lambda}^{-1}\|(\Ac+i\lambda)^{-1}\|_{\mathcal{L}(\mathcal{H})}. 
\end{equation}
\item Let $W$ be normally $L^1$. Then there is $C>0$ such that
\begin{equation}\label{2l23}
\|(\Ac+i\lambda)^{-1}\|_{\mathcal{L}(\mathcal{H})}\le \abs{\lambda}\|P_{\lambda}^{-1}\|_{L^2\rightarrow L^2}+\|P_\lambda^{-1}\|_{L^2\rightarrow H^1}+\left(1+C\abs{\lambda}^{-1}\langle \lambda\rangle\right) \|P_{-\bar\lambda}^{-1}\|_{L^2\rightarrow H^{1}}+C\abs{\lambda}^{-1},
\end{equation}
uniformly for any $\lambda\in \{\cim\lambda\ge0, \lambda\neq 0\}$. 
\item If we further assume $W$ is normally $L^p$ for $p\in (1,\infty]$, then there exists $K>0$ such that \eqref{2l23} is also uniformly true for any $\lambda\in \{\cim\lambda<-K\}$ at which both $P^{-1}_\lambda$ and $P^{-1}_{-\bar\lambda}$ are bounded from $L^2$ to $H^1$. 
\end{enumerate}
\end{lemma}
\begin{proof}
1. Note that $(\Ac+i\lambda)(-u, i\lambda u)=(0, P_\lambda u)$. This implies
\begin{equation}
\|u\|_{H^1}\le \|(\Ac+i\lambda)^{-1}\|_{\mathcal{L}(\mathcal{H})}\|P_\lambda u\|_{L^2}, \ \|u\|_{L^2}\le \abs{\lambda}^{-1}\|(\Ac+i\lambda)^{-1}\|_{\mathcal{L}(\mathcal{H})}\|P_\lambda u\|_{L^2}. 
\end{equation}

2. The normally $L^1$ and $L^p$ assumptions will allow us to use the part 3 of Propositions \ref{maplemma} and \ref{2t6} respectively on each domain. Firstly note 
\begin{equation}
\left\|P^{-1}_\lambda\right\|_{H^{-1}\rightarrow L^2}=\left\|\left(P^{-1}_\lambda\right)^*\right\|_{L^2\rightarrow H^1}=\left\|\left(P^*_\lambda\right)^{-1}\right\|_{L^2\rightarrow H^1}=\left\|P_{-\bar\lambda}^{-1}\right\|_{L^2\rightarrow H^1}. 
\end{equation}
Now let $(\Ac+i\lambda)(u,v)=(f,g)$ and \eqref{2l25} implies
\begin{equation}
u=-i\lambda^{-1}f-i\lambda^{-1}P_\lambda^{-1}\Delta f-P_\lambda^{-1}g, \ v=-P_\lambda^{-1}\Delta f+i\lambda P_\lambda^{-1}g.
\end{equation}
Note
\begin{equation}\label{2l26}
\|P_\lambda^{-1}\Delta f\|_{L^2}\le \|P_{\lambda}^{-1}\|_{H^{-1}\rightarrow L^2}\|f\|_{H^1}=\|P_{-\bar\lambda}^{-1}\|_{L^2\rightarrow H^{1}}\|f\|_{H^1}. 
\end{equation}
Thus
\begin{equation}
\|v\|_{L^2}\le \left(\|P_{-\bar\lambda}^{-1}\|_{L^2\rightarrow H^{1}}+\abs{\lambda}\|P_{\lambda}^{-1}\|_{L^2\rightarrow L^2}\right)\|(f,g)\|_{\mathcal{H}}.
\end{equation}
On another hand, apply the part 3 of Propositions \ref{maplemma} and \ref{2t6} to \eqref{2l26} to see
\begin{equation}
\|P_\lambda^{-1}\Delta f\|_{H^1}\le C\langle \lambda\rangle \|P_\lambda^{-1}\Delta f\|_{L^2}+C\|\Delta f\|_{H^{-1}}\le C\left(\langle \lambda\rangle \|P_{-\bar\lambda}^{-1}\|_{L^2\rightarrow H^{1}}+1\right)\|f\|_{H^1}. 
\end{equation}
Then 
\begin{equation}
\|u\|_{H^1}\le \left(\|P_\lambda^{-1}\|_{L^2\rightarrow H^1}+C\abs{\lambda}^{-1}\langle \lambda\rangle \|P_{-\bar\lambda}^{-1}\|_{L^2\rightarrow H^{1}}+C\abs{\lambda}^{-1}\right)\|(f,g)\|_{\mathcal{H}}. 
\end{equation}
Bring together the estimates for $u$ and $v$ to conclude. 
\end{proof}

We cite a backward uniqueness result for semigroups and use it to prove our semigroup $e^{t\Ac}$ is backward unique when $W$ is normally $L^p$:
\begin{proposition}[Lasiecka-Renardy-Triggiani, Theorem 3.1 of \cite{lrt01}]\label{2t17}
Let $\Ac$ generates a strongly continuous semigroup $e^{t\Ac}$ on a Banach space $X$. Assume there exist $\theta\in (\pi/2,\pi), R,C>0$ such that uniformly for all $r\ge R$ we have
\begin{equation}
\|(\Ac-re^{\pm i\theta})^{-1}\|_{\mathcal{L}(X)}\le C.
\end{equation}
Then $e^{t\Ac}$ is backward unique, that is, if $e^{t\Ac}x=0$ at some $t>0$, then $x=0$. 
\end{proposition}

\begin{proposition}[Backward uniqueness for $e^{t\Ac}$]\label{2t7}
Let $W$ be normally $L^p$ for $p\in(1,\infty]$. Then $e^{t\Ac}$ is backward unique, that is, if $e^{t\Ac}(u,v)=0$ at some $t>0$, then $(u,v)=0$. 
\end{proposition}
\begin{proof}
1. We claim for any $\theta$ such that $\theta\in (\pi/2,3\pi/2)$, along the ray $-i\lambda=\abs{\lambda}e^{i\theta}$, there exists $C,\lambda_0>0$ such that for all $\abs{\lambda}>\lambda_0$, we have $\|(\Ac+i\lambda)^{-1}\|_{\mathcal{L}(\mathcal{H})}\le C\abs{\lambda}^{-1}$. Note that $\arg(\lambda)\in (\pi,2\pi)$ and so is $\arg(-\bar\lambda)$. From Proposition \ref{2t6}, we have 
\begin{equation}
\|P_{\lambda}^{-1}\|_{L^2\rightarrow H^1}\le C\abs{\lambda}^{-1}, \|P_{\lambda}^{-1}\|_{L^2\rightarrow L^2}\le C\abs{\lambda}^{-2}, 
\end{equation}
and 
\begin{equation}
\|P_{-\bar\lambda}^{-1}\|_{L^2\rightarrow H^1}\le C\abs{\lambda}^{-1},
\end{equation}
for $\abs{\lambda}\ge \lambda_0$. Invoke Lemma \ref{2t8} to see $\|(\Ac+i\lambda)^{-1}\|_{\mathcal{L}(\mathcal{H})}\le C\abs{\lambda}^{-1}$. 

2. Now invoke Proposition \ref{2t17} to conclude the backward uniqueness. 
\end{proof}

\subsection{Semigroup decomposition}\label{s2.3}
In this subsection, we always assume $W\neq 0$. We would like to apply the following results to $e^{t\Ac}$, if possible. 
\begin{proposition}[Borichev-Tomilov, Theorem 2.4 of \cite{bt10}]\label{BTthm}
Let $e^{t\Acd}$ be a strongly continuous semigroup on a Hilbert space $\Hcd$, generated by $\Acd$. If $i \Rb \cap\spec(\Acd)=\emptyset$, then the following conditions are equivalent:
\begin{align}
\| e^{t \Acd} \Acd^{-1} \|_{\Lc(\Hcd)} &= O(\langle t\rangle^{-\alpha}) & \text{as } t \ra \infty, \\
\| (i\lambda Id -\Acd)^{-1} \|_{\Lc(\Hcd)} &= O(|\lambda|^{1/\alpha}) &\text{as } \lambda \ra \infty. 
\end{align}
\end{proposition}

\begin{proposition}[Gearhart-Pr\"{u}ss-Huang, \cite{gea78,pru84,hua85}]\label{GPHthm}
Let $e^{t \Acd}$ be a strongly continuous semigroup on a Hilbert space $\Hcd$ and assume that there exists a positive constant $M>0$ such that $\| e^{t\Acd}\| \leq M$ for all $t \geq 0$. Then there exist $C, c>0$ such that for all $t>0$
\begin{equation}
\|e^{t\Acd}\|_{\Lc(\Hcd)} \leq C e^{-ct},
\end{equation}
if and only if $i \Rb \cap\spec(\Acd)=\emptyset$ and 
\begin{equation}    
\sup_{\lambda \in \Rb} \|(\Acd-i\lambda Id)^{-1} \|_{\Lc(X)} < \infty.
\end{equation}
\end{proposition}
An issue with applying these is that $\Ac$ can have spectrum at $0$. The outline for Section \ref{s2.3} is to define a semigroup generator $\Acd$ with no spectrum at 0 and which provides energy decay information for $e^{t\Ac}$. We then establish an equivalence of resolvent estimates for $\Acd$ and $\Pl$, which we use to prove Propositions \ref{polyprop} using Propositions \ref{BTthm} and \ref{GPHthm}. We follow the strategy of \cite{aln14} to separate the zero-frequency modes from others. 

\begin{definition}[Spectral decomposition]
Assume $W\neq 0$, then $\Ac$ has a simple pole at $0$. Let $\Pi_0$ be the Riesz projector of $\Ac$ that projects $\mathcal{H}$ onto $\ker \Ac$, given by the spectral resolution
\begin{equation}
\Pi_0=\frac{1}{2\pi i} \int_{\gamma} (z \id -\Ac)^{-1} dz,
\end{equation}
where $\gamma$ is a small circle around $0$ in $\C$ containing $0$ as the only eigenvalue of $\Ac$ in its interior. Consider the range of $\Ac$, 
\begin{equation}
\Hcd=\Ac(\Dc)=\{(u,v)\in \Hc: \av(Wu+v)=0\},
\end{equation}
a codimension-1 subspace of $\mathcal{H}$, equipped with the norm $\|(u, v)\|_{\Hcd}^2 := \ltwo{\nabla u}^2 + \ltwo{v}^2$. Note $\Pi_\bullet=\id-\Pi_0$ projects $\mathcal{H}$ onto $\Hcd$. The Riesz projectors non-orthogonally decomposes $\mathcal{H}$ as $\Pi_0 \Hc\oplus\Hcd$. Concretely, let $\abs{M}$ be the volume of $M$ and for any $(u,v)\in \mathcal{H}$, the Riesz projectors are 
\begin{equation}
\Pi_0 (u,v)=(\av(W)^{-1}\av (W u+v), 0), \ \Pi_\bullet(u,v)=(u-\av(W)^{-1}\av (W u+v), v),
\end{equation}
where $\av (u)=\frac{1}{\abs{M}}\int_M u\ d M$, the average of $u$ over $M$, and $\av(W)>0$ since $W\neq 0$. Note $\Pi_0\Pi_\bullet=\Pi_\bullet\Pi_0=0$, and thus $\Hcd=\ker \Pi_0$, $\Pi_0 \mathcal{H}=\ker \Ac=\ker\Pi_\bullet=\ker \Delta\times \{0\}=\spa \{(1,0)\}$, and $\|\Pi_\bullet (u,v)\|_{\Hcd}=\|(u,v)\|_{\Hcd}$ for any $(u,v)\in \mathcal{H}$. Let
\begin{equation}
\Acd=\Ac|_{D(\Acd)}=
\begin{pmatrix}
0 & \id\\
\Delta & -W
\end{pmatrix}:D(\Acd)\rightarrow\Hcd,
\end{equation}
where $\Dcd=D(\Acd)=\Pi_\bullet D(\Ac)=\Dc\cap \Hcd$, equipped with the norm
\begin{equation}
\|(u,v)\|_{D(\Acd)}^2=\|\nabla u\|_{L^2}^2+\|v\|_{H^1}^2+\|-\Delta u+Wv\|_{L^2}^2.
\end{equation}
Then $\Dc=\Dcd\oplus \Pi_0 \Dc$ and $\Pi_0 \Dc=\Pi_0 \mathcal{H}=\ker \Ac$. Note that the nature of spectral resolution implies $\Pi_\bullet\Ac=\Ac \Pi_\bullet=\Acd \Pi_\bullet$, $\Pi_0\Ac=\Ac\Pi_0=0$ and thus $\Ac=\Acd\Pi_\bullet$. 
\end{definition}

\begin{remark}
Without assuming $W\neq 0$, we can always orthogonally decompose $\Hc=\ker\Ac\oplus (\ker \Ac)^\perp$. Doing this will not invalidate most of the theorems in this section nor the main results, but we point out that it is natural to use the non-orthogonal decomposition $\ker\Ac\oplus \Hcd$. When $\Ac$ is not normal, we do not expect its eigenspaces with  distinct eigenvalues to be  orthogonal to each other. In such a case, the lack of orthogonality of $\Pi_0$ is a known phenomenon: see further in \cite[Proposition 6.3]{hs96}, \cite[Remarks(1), p. 84]{dz19}, \cite[Proposition 50.2]{heu82}. 
\end{remark}

\begin{lemma}[Spectral and resolvent equivalence between $\Acd$ and $\Ac$]\label{2t11}
Let $\mu\in \mathbb{C}\setminus \{0\}$, $W\neq 0$. Then there exists $C>0$ independent of $\mu$ such that the following are true:
\begin{enumerate}
\item $\Acd:\Dcd\rightarrow\Hcd$ is bijective. 
\item $\Acd-\mu:\Dcd\rightarrow\Hcd$ is bijective if and only if $\Ac-\mu:\Dc\rightarrow\Hc$ is bijective. 
\item On $\Hcd$, the $\Hcd$-norm and $\Hc$-norm are equivalent. 
\item If both $\Acd-\mu$ and $\Ac-\mu$ are bijective, then 
\begin{equation}
C^{-1}\|(\Acd-\mu)^{-1}\|_{\mathcal{L}(\Hcd)}\le \|(\Ac-\mu)^{-1}\|_{\mathcal{L}(\Hc)}\le C\left(\|(\Acd-\mu)^{-1}\|_{\mathcal{L}(\Hcd)}+\abs{\mu}^{-1}\right).
\end{equation}
\end{enumerate}
\end{lemma}
\begin{proof}
1. Note that from Lemma \ref{2t9} we know $\Ac:\Dc=\Dcd\oplus \ker \Ac\rightarrow \Hcd$ is surjective and thus $\Acd:\Dcd\rightarrow \Hcd$ is bijective. Note that $\|(u,v)\|_{\Hcd}\le \|(u,v)\|_{\Hc}$ and $(\Hcd,\|\cdot\|_\Hc)$ is continuously embedded in $(\Hcd,\|\cdot\|_{\Hcd})$. Since the embedding is a bijective continuous map, it is further an open map and admits a continuous inverse. This implies that the norms are equivalent and we write $\|\cdot\|_{\Hcd}\le \|\cdot\|_{\Hc}\le C\|\cdot\|_{\Hcd}$. 

2. Let $\mu\in\mathbb{C}\setminus \{0\}$. Assume $\Acd-\mu:\Dcd\rightarrow\Hcd$ is bijective. Note that $\Pi_\bullet \Ac=\Acd \Pi_\bullet$. Then consider the decomposition
\begin{equation}\label{2l27}
\Ac-\mu=(\Pi_\bullet+\Pi_0)\Ac-\mu=(\Acd-\mu)\Pi_\bullet+\Pi_0(\Ac-\mu).
\end{equation}
For any $V\in \mathcal{H}$, since $\Acd-\mu$ is surjective, there exists $\dot U\in \Dcd$ such that $(\Acd-\mu)\dot U=\Pi_\bullet V$. Then
\begin{equation}
(\Ac-\mu)\left(\dot U-\mu^{-1}\Pi_0 V\right)=(\Acd-\mu) \dot U-\Pi_0(\Ac-\mu) \mu^{-1}\Pi_0 V=\Pi_\bullet V+\Pi_0 V=V,
\end{equation}
where we used $\Ac\Pi_0=0$. Thus $\Acd-\mu$ is surjective. To show it is injective, assume $(\Ac-\mu)U=0$ for some $U\in \Dc$. Then
\begin{equation}
0=\Pi_\bullet(\Ac-\mu)U=(\Acd-\mu)\Pi_\bullet U.
\end{equation}
As $(\Acd-\mu)$ is injective, $\Pi_\bullet U=0$. Thus $U=\Pi_0 U$. Then
\begin{equation}
0=\Pi_0(\Ac-\mu)U=-\mu U. 
\end{equation}
As $\mu\neq 0$, $U=0$ and $\Ac-\mu$ is also injective. 

3. Assume $\Ac-\mu:\Dc\rightarrow\Hc$ is bijective from $\Dc=\Dcd\oplus \Pi_0 \Dc$ to $\mathcal{H}=\Hcd \oplus \Pi_0 H$. Note that $\Ac=0$ on $\Pi_0 \Dc$ and $\Ac-\mu$ maps $\Pi_0 \Dc$ to $\Pi_0 \Dc=\Pi_0 H$ bijectively. Thus $\Ac-\mu$ is bijective from $\Dcd$ to $\dot H$. Eventually observe that $\Acd-\mu=\Pi_\bullet (\Ac-\mu)$ on $\Hcd$ to conclude it is bijective. 

4. Now assume both $\Acd-\mu:\Dcd\rightarrow\Hcd$ and $\Ac-\mu:\Dc\rightarrow\Hc$ are bijective. Fix $\dot V\in\Dcd$, and let $\dot U$ be the unique element in $\Hcd$ such that $(\Acd-\mu) \dot U=\dot V$. Since $\dot V\in \Dc$, there exists a unique $U\in \Hc$ such that $(\Ac-\mu) U=\dot V$. Moreover,
\begin{equation}
(\Acd-\mu)\Pi_\bullet U=\Pi_\bullet (\Ac-\mu) U=\Pi_\bullet \dot V=\dot V,
\end{equation}
implies $\dot U=\Pi_\bullet U$. Thus
\begin{equation}
\|\dot U\|_{\Hcd}\le \|\dot U\|_{\Hc}\le C\|U\|_{\Hc}\le C\|(\Ac-\mu)^{-1}\|_{\mathcal{L}(\Hc)}\|\dot V\|_{\Hc}\le C \|(\Ac-\mu)^{-1}\|_{\mathcal{L}(\Hc)}\|\dot V\|_{\Hcd}
\end{equation}
and $\|(\Acd-\mu)^{-1}\|_{\mathcal{L}(\Hcd)}\le C\|(\Ac-\mu)^{-1}\|_{\mathcal{L}(\Hc)}$. On another hand, let $(\Ac-\mu)U=V$ for $U\in \Dc$, $V\in \mathcal{H}$. Then
\begin{equation}
(\Acd-\mu)\Pi_\bullet U= \Pi_\bullet (\Ac-\mu)U=\Pi_\bullet V,
\end{equation}
implies $\|\Pi_\bullet U\|_{\mathcal{H}}\le \|(\Acd-\mu)^{-1}\|_{\mathcal{L}(\Hcd)}\|\Pi_\bullet V\|_{\mathcal{H}}$. Meanwhile,
\begin{equation}
\Pi_0 V=\Pi_0(\Ac-\mu)U=-\mu\Pi_0U,
\end{equation}
where we used $\Pi_0\Ac=0$. Thus $\abs{\mu} \|\Pi_0 U\|_{\mathcal{H}}\le \|\Pi_0 V\|_\mathcal{H}$, and
\begin{equation}
\|U\|_{\mathcal{H}}\le C\left(\|(\Acd-\mu)^{-1}\|_{\mathcal{L}(\Hcd)}+\abs{\mu}^{-1}\right)\|V\|_{\mathcal{H}},
\end{equation}
as we want. 
\end{proof}

We now show that this $\Acd$ indeed generates a contraction semigroup that provides energy decay information about $e^{t\Ac}$. 
\begin{proposition}[Semigroup decomposition]\label{2t12}
Let $W$ be normally $L^1$ and $W\neq 0$. Then the following are true:
\begin{enumerate}[wide]
\item The generator $\Acd:\Dcd\rightarrow \Hcd$ is maximally dissipative.
\item The generator $\Acd$ generates a contraction semigroup $e^{t\Acd}$ on $\Hcd$.
\item The strongly continuous semigroup generated by $\Ac$ on $\Hc$ can be decomposed as
\begin{equation}\label{2l1}
e^{t\Ac}=e^{t\Acd} \Pi_\bullet+\Pi_0,
\end{equation}
\item We have $\Pi_\bullet e^{t\Ac}=e^{t\Acd}\Pi_\bullet$.
\item There exists $C>0$ such that $\|e^{t\Ac}\|_{\mathcal{L}(\Hc)}\le C$ for all $t\ge 0$. 
\end{enumerate}
\end{proposition}
\begin{proof}
1. We show $\Acd: \Dcd\rightarrow \Hcd$ is maximally dissipative. Firstly note that Corollary \ref{2t10} implies $\Ac-1$ is bijective from $\Dc=\Dcd\oplus \Pi_0 \Dc$ to $\mathcal{H}=\Hcd \oplus \Pi_0 H$. Apply Lemma \ref{2t11} to conclude that $\Acd -1$ is bijective from $\Dcd$ to $\Hcd$. Then it is straightforward to compute for any $(u,v)\in \Dcd$, 
\begin{equation}
\cre{\left\langle \Acd(u,v), (u,v)\right\rangle_{\Hcd}}=-\int_M W \abs{v}^2\le 0,
\end{equation}
which demonstrates the dissipative nature of $\Acd$, and so by the Lumer-Phillips theorem as in \cite[Theorem II.3.15, p.83]{en99}, $\Acd$ generates a contraction semigroup $e^{t\Acd}$ on $\Hcd$, and $\|e^{t\Acd}\|_{\Hcd}\le 1$ for all $t\ge 0$. 

2. We begin with the initial decomposition
\begin{equation}
e^{t\Ac}=e^{t\Ac}\Pi_\bullet+e^{t\Ac}\Pi_0,
\end{equation}
and simplify both terms. Consider for the first term that on $\Hcd$, 
\begin{equation}
\partial_t e^{t\Ac}=\Ac=\Acd, \ e^{t\Ac}|_{t=0}=\id,
\end{equation}
Thus $\Acd$ generates $e^{t\Ac}$: by the uniqueness of the generator we know $e^{t\Ac}=e^{t\Acd}$ on $\Hcd$. On another hand, for the third term, on $\Pi_0\mathcal{H}=\ker \Ac$, we have
\begin{equation}
\partial_t e^{t\Ac}=\Ac=0, \ e^{t\Ac}|_{t=0}=\id,
\end{equation}
and thus $e^{t\Ac}\Pi_0=\Pi_0$. The above observations give the desired decomposition \eqref{2l1}. Apply $\Pi_\bullet$ to \eqref{2l1} and note $\Pi_\bullet\Pi_0=0$ to obtain $\Pi_\bullet e^{t\Ac}=e^{t\Acd}\Pi_\bullet$.

3. For the boundedness of $e^{t\Ac}$ on $\Hc$, note for any $U\in \Hc$, 
\begin{equation}
\|e^{t\Ac}U\|_{\Hc}\le C\|e^{t\Acd} \Pi_\bullet U\|_{\Hcd} +C\|\Pi_0 U\|_{\Hc}\le C\|U\|_{\Hc},
\end{equation}
where we used $\Hcd$-norm and $\Hc$-norm are equivalent on $\Hcd$, and $\|e^{t\Acd}\|_{\mathcal{L}(\Hcd)}\le 1$. 
\end{proof}

\begin{proposition}[Backward uniqueness for $e^{t\Acd}$]\label{2t13}
Let $W$ be normally $L^p$ for $p\in(1,\infty]$. Then $e^{t\Acd}$ is backward unique, that is, if $e^{t\Acd}(u,v)=0$ at some $t>0$, then $(u,v)=0$. 
\end{proposition}
\begin{proof}
Similar to Proposition \ref{2t7}: Lemma \ref{2t11} implies on the spectral region of interest, the resolvents $\|(\Ac+i\lambda)^{-1}\|_{\mathcal{L}(\Hc)}$ and $\|(\Acd+i\lambda)^{-1}\|_{\mathcal{L}(\Hcd)}$ obey the same bounds. 
\end{proof}

\begin{proof}[Proof of Theorem \ref{backwardthm}]
Part 1 immediately follows from Proposition \ref{2t7}. To prove the backward uniqueness for the energy, assume $E(u,t)=0$ at some $t>0$. This implies
\begin{equation}
0=\|e^{t\Ac}(u_0, u_1)\|_{\Hcd}=\|\Pi_\bullet e^{t\Ac}(u_0, u_1)\|_{\Hcd}=\|e^{t\Acd}\Pi_\bullet (u_0, u_1)\|_{\Hcd},
\end{equation}
where we used that $\|\Pi_\bullet (u,v)\|_{\Hcd}=\|(u,v)\|_{\Hcd}$ for any $(u,v)\in \mathcal{H}$ and Proposition \ref{2t12}. Proposition \ref{2t13} implies $\Pi_\bullet (u_0,u_1)=0$, and thus
\begin{equation}
E(u,0)^{\frac12}=\frac{1}{2}\|(u_0, u_1)\|_{\Hcd}=\frac{1}{2}\|\Pi_\bullet(u_0, u_1)\|_{\Hcd}=0,
\end{equation}
as we want in Part 2. 
\end{proof}

\begin{lemma}[Resolvent equivalence on the real line]\label{2t1}
Let $\alpha\ge -1$ and $\lambda_0> 0$. The following are equivalent:
\begin{enumerate}[wide]
\item There exists $C>0$ such that for all $\lambda\in\mathbb{R}$ with $\abs{\lambda}\ge \lambda_0$ we have
\begin{equation}\label{2l10}
\|P_\lambda^{-1}\|_{L^2\rightarrow L^2}\le C\langle \lambda\rangle^{\alpha}.
\end{equation}
\item There exists $C>0$ such that for all $\lambda\in\mathbb{R}$ with $\abs{\lambda}\ge \lambda_0$ we have
\begin{equation}\label{2l11}
\| (\Acd+i\lambda)^{-1}\|_{\mathcal{L}(\Hcd)} \leq C\langle \lambda\rangle^{\alpha+1}.
\end{equation}
\end{enumerate}
\end{lemma}
\begin{proof}
1. Assume \eqref{2l11} holds. Lemma \ref{2t11} implies $\|(\Ac+i\lambda)^{-1}\|_{\mathcal{L}(\Hc)}\le C\langle \lambda\rangle^{\alpha+1}$ and Lemma \ref{2t8}(1) implies \eqref{2l10}. 

2. Assume \eqref{2l10} holds. Proposition \ref{maplemma}(3) implies
\begin{equation}
\|P_\lambda^{-1}\|_{L^2\rightarrow H^1}\le C\langle \lambda\rangle^{\alpha+1}+C\le C\langle \lambda\rangle^{\alpha+1},
\end{equation}
uniformly for all $\lambda\in \{\lambda\in \mathbb{R}:\abs{\lambda}\ge \lambda_0\}$. For $\lambda\in \{\lambda\in \mathbb{R}:\abs{\lambda}\ge \lambda_0\}$, $-\bar\lambda=-\lambda$ is also in this set. Lemma \ref{2t8} implies
\begin{equation}
\|(\Ac+i\lambda)^{-1}\|_{\mathcal{L}(\Hc)}\le \abs{\lambda}\langle \lambda\rangle^{\alpha}+(2+C\abs{\lambda}^{-1}\langle \lambda\rangle)(C\langle \lambda\rangle^{\alpha+1}+C)+C\abs{\lambda}^{-1}\le C\langle \lambda\rangle^{\alpha+1}. 
\end{equation}
Invoke Lemma \ref{2t11} to see \eqref{2l11}. 
\end{proof}
We now give full proof to Proposition \ref{polyprop}, that resolvent estimates of $P_\lambda$ are equivalent to energy decay. 
\begin{proof}[Proof of Proposition \ref{polyprop}]
1. We assume $\Pli$, for $|\lambda| \geq \lambda_0$ 
\begin{equation}
\|\Pli\|_{L^2 \ra L^2} \leq C \abs{\lambda}^{1/\alpha-1},
\end{equation}
by Lemma \ref{2t1} this is equivalent to 
\begin{equation}
\|(\Acd+i\lambda)^{-1}\|_{\mathcal{L}(\Hcd)} \leq C \langle \lambda\rangle^{1/\alpha},
\end{equation}
for $|\lambda| \geq \lambda_0$. Note that Corollary \ref{2t10}(1) with Lemma \ref{2t11} implies $i\mathbb{R}\cap\spec{}\Acd=\emptyset$. By Proposition \ref{BTthm} of Borichev-Tomilov, this is equivlent to 
\begin{equation}
\|e^{t \Acd} \Acd^{-1}\|_{\Lc(\Hcd)} = O(\langle t\rangle^{-\alpha}). 
\end{equation}
This is equivalent to that, the energy of solution $u$ to the damped wave equation \eqref{DWE} is bounded by
\begin{multline}
E(u, t)^{\frac12}\le C\left\|e^{t\Ac}(u_0, u_1)\right\|_{\Hcd}=C\left\|e^{t\Acd}\Pi_\bullet(u_0, u_1)\right\|_{\Hcd}
\le \left\|e^{t\Acd}\Acd^{-1}\right\|_{\mathcal{L}(\Hcd)}\left\|(u_0, u_1)\right\|_{\Dcd}\\
\le C\langle t\rangle^{-\alpha}\left(\|\nabla u_0\|_{L^2}+\|u_1\|_{H^1}+\|-\Delta u_0+Wu_1\|_{L^2}\right),
\end{multline}
as desired. 

2. When $\|P_\lambda^{-1}\|_{L^2\rightarrow L^2}\le C$, we apply Proposition \ref{GPHthm} of Gearhart-Pr\"uss-Huang.
\end{proof}

\subsection{Schr\"odinger observability gives polynomial decay}
Let $e^{-it\Delta}$ be the unitary Schr\"odinger operator group on $L^2$ generated by the anti-self-adjoint operator $-i\Delta: H^2\rightarrow L^2$. The Schr\"odinger equation
\begin{equation}
	(i\partial_t-\Delta)u=0, \ u|_{t=0}=u_0\in L^2,
\end{equation} 
is unique solved by $u=e^{-it\Delta}u_0$. 
\begin{definition}[Schr\"odinger observability]\label{observability}
We say that the Schr\"odinger equation is exactly observable from an open set $\Omega\subset M$ if there exists $T>0, C_T>0$ such that for any $u_0 \in L^2$, 
\begin{equation}
\|u_0\|_{L^2}\le C_T\int_{0}^{T} \|\mathbbm{1}_{\Omega}e^{-it\Delta} u_0\|_{L^2}\ dt.
\end{equation}
\end{definition}
We now consider the spectral theory of $-\Delta$. Since $-\Delta$ is essentially self-adjoint and positive-definite on $L^2$, we have a spectral resolution
\begin{equation}
-\Delta u=\int_0^\infty\rho^2\ dE_\rho(u),
\end{equation}
where $E_\rho$ is a projection-valued measure on $L^2$ and $\supp E_\rho\subset [0, \infty)$. Define the scaling operators 
\begin{equation}
\Lambda^{-s}=\int_0^\infty (1+\rho^2)^{-s}\ dE_\rho(u). 
\end{equation}
Those operators $\Lambda^{-s}: H^{-s}\rightarrow L^2$ are elliptic and bounded from above and below, and they commute with $-\Delta$. 
\begin{lemma}\label{2t15}
Fix $s, N>0$. Assume the Schr\"odinger equation is exactly observable from $\Omega$. Let $\chi\in C^\infty$ where $\chi= 1$ on a open neighbourhood of $\overline\Omega$. Then there is $C>0$ such that 
\begin{equation}
\|u\|_{H^{-s}}\le C\|(-\Delta-\lambda^2)u\|_{H^{-s}}+C\|\chi u\|_{L^2}^2+C\|u\|_{H^{-N}}^2.
\end{equation}
for all $\lambda\in\mathbb{R}$.
\end{lemma}
\begin{proof}
When the Schr\"odinger equation is exactly observable from $\Omega$, \cite[Theorem 5.1]{mil05} implies for any $v\in L^2$
\begin{equation}\label{2l28}
\|v\|_{L^2}\le C\|(-\Delta-\lambda^2)v\|_{L^2}+C\|\mathbbm{1}_\Omega v\|_{L^2}.
\end{equation}
Now let $u\in H^{-s}$ and $v=\Lambda^{-s}u\in L^2$. Apply \eqref{2l28} to see
\begin{equation}
\|\Lambda^{-s}u\|_{L^2}\le C\|(-\Delta-\lambda^2)\Lambda^{-s}u\|_{L^2}+C\|\mathbbm{1}_\Omega \Lambda^{-s} u\|_{L^2}=C\|\Lambda^{-s}(-\Delta-\lambda^2)u\|_{L^2}+C\|\mathbbm{1}_\Omega \Lambda^{-s}u\|_{L^2},
\end{equation}
which implies
\begin{equation}
\|u\|_{H^{-s}}\le C\|(-\Delta-\lambda^2)u\|_{H^{-s}}+C\|\mathbbm{1}_\Omega\Lambda^{-s} u\|_{L^2}.
\end{equation}
Now fix a cutoff $\chi_0$ such that $\chi_0\equiv 1$ on $\Omega$, $\supp{\chi_0}\subset \supp \chi$ compactly, and $\chi\chi_0=\chi_0$. Note that since $\WF{(\chi_0\Lambda^{-s})} \subset \Ell{\chi}$, we have
\begin{equation}
\|\mathbbm{1}_\Omega\Lambda^{-s} u\|_{L^2}^2\le \|\chi_0\Lambda^{-s} u\|_{L^2}^2\le C \|\chi u\|_{H^{-s}}^2+C\|u\|_{H^{-N}}^2,
\end{equation}
for any $N>0$ from the elliptic estimate \cite[Theorem E.33]{dz19}. 
\end{proof}

\begin{proposition}\label{2t19}
Let $W$ be normally $L^p$ for $p\in (1,\infty)$. Assume the Schr\"odinger equation is exactly observable from $\Omega$, and there exists $\epsilon>0$ such that $W\ge \epsilon$ almost everywhere on an open neighbourhood of $\overline{\Omega}$. Then there exists $\lambda_0, C>0$ such that for all $|\lambda| \geq \lambda_0$
\begin{equation}
\|P_\lambda^{-1}\|_{L^2\rightarrow L^2}\le C\langle \lambda\rangle^{1+\frac{1}{p}}.
\end{equation}
When $p=1$, $\|P_\lambda^{-1}\|_{L^2\rightarrow L^2}\le C\langle \lambda\rangle^{2+}$. When $p=\infty$, $\|P_\lambda^{-1}\|_{L^2\rightarrow L^2}\le C\langle \lambda\rangle$ uniformly for all $\abs{\lambda}\ge \lambda_0$.
\end{proposition}
\begin{proof}
1. Let $p\in (1,\infty)$ and $s=\frac{1}{2p}$. Let $P_{\lambda}u=(-\Delta-i\lambda W-\lambda^2)u=f$. There exists $\chi\in C^\infty$ such that $\chi=1$ on $\overline\Omega$, while $\supp \chi$ is compactly supported in  $\{W\ge \epsilon\}$. Since on $\supp\chi$, $W$ is bounded from below, we have $\|\chi u\|_{L^2}\le \|\chi\sqrt{W}u\|_{L^2}$. Lemma \ref{2t15} then implies
\begin{equation}
\|u\|_{H^{-s}}^2\le C\|(-\Delta-i\lambda W-\lambda^2)u\|_{H^{-s}}^2+\lambda^2\|Wu\|_{H^{-s}}^2+C\|\chi\sqrt{W} u\|_{L^2}^2+C \|u\|_{H^{-N}}^2.
\end{equation}
Since $\sqrt{W}:L^2\rightarrow H^{-s}$ is bounded, we have
\begin{equation}\label{2l30}
\|u\|_{H^{-s}}^2\le C\|f\|_{H^{-s}}^2+C\lambda^2\|\sqrt{W} u\|_{L^2}^2+C \|u\|_{H^{-N}}^2.
\end{equation}
We now get rid of the last term on the right. Pair $P_\lambda u$ with $u$ in $H^{-N}$ to observe
\begin{equation}
\lambda^{2}\|u\|_{H^{-N}}^2=\|\nabla u\|_{H^{-N}}^2-i\lambda\langle Wu, u\rangle_{H^{-N}} -\langle f, u\rangle_{H^{-N}}.
\end{equation}
This implies
\begin{equation}\label{2l29}
\|u\|_{H^{-N}}^2\le C\lambda^{-2}\|u\|_{H^{1-N}}^2+C\abs{\lambda}^{-1}\abs{\langle Wu, u\rangle_{H^{-N}}}+C\lambda^{-2}\|f\|_{H^{-N}}^2. 
\end{equation}
Note that $\langle Wu, u\rangle_{H^{-N}}=\langle \Lambda^{-N} \sqrt{W}\sqrt{W} u , \Lambda^{-N} u \rangle$. When $N\ge s$, $\Lambda^{-N}\sqrt{W}$ is a bounded map on $L^2$. Thus we have
\begin{equation}
\abs{\langle Wu, u\rangle_{H^{-N}}}\le C\epsilon^{-1}\|\sqrt{W}u\|_{L^2}^2+\epsilon \|u\|_{H^{-N}}^2,
\end{equation}
for any $N\ge s$. Now fix $N>s+1$, \eqref{2l29} implies
\begin{equation}
\|u\|_{H^{-N}}^2\le C\lambda^{-2}\|u\|_{H^{-s}}+C\abs{\lambda}^{-1}\|\sqrt{W}u\|_{L^2}^2+C\lambda^{-2}\|f\|_{H^{-s-1}}.
\end{equation}
Applying this to \eqref{2l30} implies that for large $\abs{\lambda}\ge \lambda_0$, 
\begin{equation}
\|u\|_{H^{-s}}^2\le C\left(\|f\|_{H^{-s}}^2+\lambda^2\|\sqrt{W} u\|_{L^2}^2\right).
\end{equation}
Now pair $P_\lambda u$ with $u$ in $H^{-s}$ to observe
\begin{equation}
\|\nabla u\|_{H^{-s}}^2-\lambda^2\|u\|_{H^{-s}}^2-i\lambda\langle Wu, u\rangle_{H^{-s}}=\langle u,f\rangle_{H^{-s}}.
\end{equation}
Note $\abs{\langle Wu, u\rangle_{H^{-s}}}\le C\epsilon^{-1}\|\sqrt{W}u\|_{L^2}^2+\epsilon \|u\|_{H^{-s}}^2$ and thus
\begin{equation}
\|u\|_{H^{1-s}}^2\le C\langle \lambda\rangle^2\|u\|_{H^{-s}}^2+C\langle \lambda\rangle^{-2}\|f\|_{H^{-s}}^2+\abs{\lambda}\abs{\langle Wu, u\rangle_{H^{-s}}}\le C\langle \lambda\rangle^2\left(\|f\|_{H^{-s}}^2+\lambda^2\|\sqrt{W} u\|_{L^2}^2\right).
\end{equation}
Apply the interpolation inequality \eqref{2l31} to $\Lambda^{-s}u$ with $\gamma=\lambda^s$ and $r=s$ to see
\begin{equation}
\|u\|_{L^2}^2\le C\langle \lambda\rangle^{2s}\left(\|f\|_{H^{-s}}^2+\lambda^2\|\sqrt{W} u\|_{L^2}^2\right).
\end{equation}
Pair $P_\lambda u$ with $u$ in $L^2$ to observe $\|\nabla u\|^2-\lambda^2\|u\|^2-i\lambda \|\sqrt{W}u\|^2=\langle f,u\rangle$ and
\begin{equation}
\abs{\lambda} \|\sqrt{W}u\|^2\le C\epsilon^{-1}\langle \lambda\rangle^{2s+1} \|f\|^2+\epsilon\langle \lambda\rangle^{-2s-1}\|u\|^2.
\end{equation}
Then $\|u\|^2\le C\langle \lambda\rangle^{2+4s} \|f\|^2$ and $\|P_\lambda^{-1}\|_{L^2\rightarrow L^2}\le C\langle \lambda\rangle^{1+2s}$. 

2. When $p=1$, for any $s>\frac{1}{2}$, the proof above still works. When $p=\infty$, $\sqrt{W}$ is bounded on $L^2$, and the above proof works with $s=0$.
\end{proof}
\begin{proof}[Proof of Theorem \ref{controlthm}]
Proposition \ref{polyprop} gives $E^{\frac{1}{2}}\le C\langle t\rangle^{-1/(2+\frac{1}{p})}$ for $p\in(1,\infty)$, and the corresponding rates for $p=1$ or $p=\infty$.
\end{proof}
\begin{remark}\label{2t16}
Assume $p\in (1,\infty)$. By a argument similar to the proof of Proposition \ref{2t19}, observe that $\|(-\Delta+(1+i\lambda)^2)^{-1}\|_{H^{-s}\rightarrow H^s}\le C\langle \lambda\rangle^{2s-1}$. For $s=\frac{1}{2p}$, we have $\|\sqrt{W}(-\Delta+(1+i\lambda)^2)^{-1}\sqrt{W}\|_{\mathcal{L}(L^2)}\le C\langle \lambda\rangle^{2s-1}$.  Apply \cite[Proposition 3.10]{cpsst19} to obtain $\|(\Acd+i\lambda)^{-1}\|_{\Hcd}\le C\langle \lambda\rangle^{2+4s}$ which gives $E^{\frac12}\le C\langle t\rangle^{-1/(2+\frac2p)}$, which is slower than our rates given in Theorem \ref{controlthm}. 
\end{remark}

\section{Resolvent estimates in one dimension}\label{resolventsection}
Consider the equation
\begin{equation}\label{sdwe}
\Pl u=(-\Delta-i\lambda W-\lambda^2)u=f.
\end{equation}
To show $\|\Pli \|_{L^2 \ra L^2} \leq \frac{C}{g(\lambda)},$ so that Propositions \ref{polyprop} can be applied, it is enough to show that there exist $C, \lambda_0 \geq 0$ such that for any $f \in L^2(M)$ and any $\abs{\lambda} \geq \lambda_0$, if $u \in H^2(M)$ solves \eqref{sdwe}, then
\begin{equation}\label{eq:resolventgoal}
\ltwo{u}^2 \leq C g(\lambda)^2 \ltwo{f}^2.
\end{equation}
To show Theorem \ref{torusthm} we must show \eqref{eq:resolventgoal} holds with $g(\lambda)^2=|\lambda|^{\frac{2}{\beta+2}}$. To show Theorem \ref{circlethm} we must show \eqref{eq:resolventgoal} holds with $g(\lambda)^2=\frac{1}{|\lambda|^2}$.

The main estimate for this section is the following one-dimensional resolvent estimate:
\begin{proposition}[1D Resolvent Estimate]\label{3t1}
Consider $u \in H^2(\Sb^1)$ that satisfies
\begin{equation}\label{SDWE}
\left(-\partial_x^2-i \lambda W-\mu^2\right)u(x)=f(x),
\end{equation}
then there is $C, \lambda_0>0$ such that for $\mu^2 \le \lambda^2$ and $\abs{\lambda}\ge \lambda_0$ we have
\begin{equation}
\left\|u\right\|_{L^2}^2+\langle \mu^2 \rangle^{-1}\|\partial_x u\|_{L^2}^2 \le C\langle \mu^2 \rangle^{-1}\left(1+\langle \mu^2 \rangle^{-1} |\lambda|^{\frac{2}{2+\beta}}\right)\left\|f\right\|_{L^2}^2.
\end{equation}
\end{proposition}
The proof of Proposition \ref{3t1} will be delayed to the second part of this section. Note that Proposition \ref{3t1} with $\mu^2=\lambda^2$ and Proposition \ref{polyprop} together imply Theorem \ref{circlethm}. Proposition \ref{3t1} can also be used to show the following proposition, which along with Proposition \ref{polyprop} implies Theorem \ref{torusthm}.

\begin{proposition}[Resolvent Estimate on Tori]\label{3t2}
Let $u\in H^2(\mathbb{T}^2)$ be the solution to 
\begin{equation}\label{3l1}
\Pl u(x,y)=\left(-\Delta -i\lambda W-\lambda^2\right)u(x,y)=f(x,y).
\end{equation}
Then there exists $C, \lambda_0>0$ such that for $\lambda\in\mathbb{R}$ with $\abs{\lambda}>\lambda_0$ we have
\begin{equation}
\left\|u\right\|_{L^2}^2+\lambda^{-2}\left\|\nabla u\right\|_{L^2}^2\le C\left(1+|\lambda|^{\frac{2}{2+\beta}}\right)\left\|f\right\|_{L^2}^2.
\end{equation}
\end{proposition}
\begin{proof}
Consider the eigenfunctions $e_n(y)$ with 
\begin{equation}
-\partial_y^2 e_n(y)=\lambda_n^2 e_n(y), \ \lambda_n^2\rightarrow \infty. 
\end{equation}
Decompose
\begin{equation}
u(x,y)=\sum u_n(x)e_n(y), \ f(x,y)=\sum f_n(x)e_n(y),
\end{equation}
and the equation \eqref{3l1} is reduced to
\begin{equation}
\left(-\partial_x^2 - i\lambda W - \mu^2 \right)u_n(x)=f_n(x), \ \mu^2=\left(\lambda^2-\lambda_n^2\right). 
\end{equation}
Apply Proposition \ref{3t1} to see uniformly in $n, \lambda$ that for $\abs{\lambda}>\lambda_0$ we have 
\begin{equation}
\left\|u_n\right\|_{L^2}^2+\langle \mu^2 \rangle^{-1}\|\nabla u_n\|_{L^2}^2 \le C\langle \mu^2 \rangle^{-1}\left(1+\langle \mu^2 \rangle^{-1} |\lambda|^{\frac{2}{2+\beta}}\right)\left\|f_n\right\|_{L^2}^2.
\end{equation}
In particular uniformly in $n, \lambda$
\begin{equation}
\|u_n\|^2\le C\left(1+|\lambda|^{\frac{2}{2+\beta}}\right)\|f_n\|^2.
\end{equation}
Apply the Parseval theorem to obtain
\begin{equation}\label{eq:ltwotorus}
\|u\|^2\le C\left(1+|\lambda|^{\frac{2}{2+\beta}}\right)\|f\|^2.
\end{equation}
Pair \eqref{3l1} with $u$ and take the real part to see
\begin{equation}
\abs{\|\nabla u\|^2-\lambda^2\|u\|^2}\le C\lambda^{-2}\|f\|^2+\frac{1}{2}\lambda^{2}\|u\|^2.
\end{equation}
This along with \eqref{eq:ltwotorus} produces the desired estimate. 
\end{proof}

The rest of this section is devoted to the proof of Proposition \ref{3t1}.
In particular, we will show that there exists $\lambda_0, \mu_0^2>0$ such that for any real $\lambda \geq \lambda_0$ and $\mu^2 \leq \lambda^2$ and $u \in H^2(\Sb^1), f \in L^2(\Sb^1)$ solving \eqref{SDWE} then
\begin{align}
\label{eq:lowenergy}
&\int |u|^2+|u'|^2 \leq C \int |f|^2&\text{ when } \mu^2<\mu_0^2, \\
\label{eq:highenergy}
&\int |u|^2 + \frac{1}{\mu^2} |u'|^2 \leq \frac{C}{\mu^2} \left( \frac{\lambda^{\frac{2}{2+\beta}}}{\mu^2} + 1 \right) \int |f|^2  &\text{ when } \mu^2 \geq \mu_0^2.
\end{align}
Here and below integrals are taken over $\Sb^1$.
The general case of $|\lambda| \geq \lambda_0$ follows by an identical argument, but we focus on $\lambda>0$ for ease of notation.  Note that in the case $\mu^2 < \mu_0^2$ we can actually have $\mu^2<0$. However, the bulk of our argument is devoted to the proof of \eqref{eq:highenergy}, where $\mu^2$ is indeed a positive real number. 

In our proof we use a version of the Morawetz multiplier method, which is arranged via the energy functional
\begin{equation}
F(x) = |u'(x)|^2 + \mu^2 |u(x)|^2.
\end{equation}
This method was introduced by \cite{mor61}. It has been used in \cite{cv02} and \cite{cd21}, we will follow its use in \cite{dk20}. 

Following the proof of Lemma 1 from \cite{dk20}, with the modification that $b$ must be chosen such that $b''<0$ on a neighborhood of each zero interval for $W$, we obtain basic estimates on the size of $u$ and $u'$ on the damped region, and \eqref{eq:lowenergy}. The proof is otherwise identical, so we do not include the details.
\begin{lemma}
If $\lambda>0, \mu^2 \in \Rb$ and $u,f$ solve \eqref{SDWE} then
\begin{equation}\label{wuestimate}
\int W |u|^2 \leq \frac{1}{\lambda} \int |fu|.
\end{equation}
Furthermore if for some $c>0, \psi \in \Cs$ is supported in $\{W>c\}$, then there exists $C>0$ such that
\begin{equation}\label{wuprime}
\int \psi|u'|^2 \leq C\left(1+ \frac{\mu^2}{\lambda}\right) \int |fu|. 
\end{equation}
Finally there are positive constants $\mu_0^2$ and $C$ such that for any $\lambda>0, \mu^2 \leq \mu_0^2$ and $u,f$ solving \eqref{SDWE} we have \eqref{eq:lowenergy}
\end{lemma} 

We now set up a multiplier, which we call $b$. The multiplier method then provides the following estimate, which must be refined to obtain our desired resolvent estimates. 
\begin{lemma}\label{resolvelemmastart}
Let $\d_j>0, j=1,\ldots, n$ be fixed constants and let 
\begin{equation}
\phi=\phi(x) = \begin{cases} \lambda^{\d_j} & |x| \in [\sigma_j-\lambda^{-\d_j}, \sigma_j+\lambda^{-\d_j}]\\
1 &\text{ otherwise}.
\end{cases}
\end{equation}
If $\lambda, \mu^2>0$ and $u,f$ solve \eqref{SDWE} then 
\begin{equation}
\int \phi |u'|^2 + \mu^2 \phi |u|^2 \lesssim \int |f|^2 + \sum_j \lambda \int W_j |uu'|.
\end{equation}
\end{lemma}
\begin{proof}
Choose $c>0$ small enough so that $\{W>c\}$ intersects the support of each $W_j$ and let $b$ be piecewise linear on $\Sb^1$ with
\begin{equation}
b'(x)= \begin{cases}
\lambda^{\d_j} & |x| \in [\sigma_j - \lambda^{-\d_j}, \sigma_j + \lambda^{-\d_j}] \\
-M & \{W>c\} \text{ and } |x| \not \in [\sigma_j - \lambda^{-\d_j}, \sigma_j + \lambda^{-\d_j}] \\
1 & \text{ elsewhere}.
\end{cases}
\end{equation}
It is possible to choose $M$ big enough so that $b$ is indeed periodic on $\Sb^1$. So then let $F=|u'|^2 + \mu^2 |u|^2$ and compute 
\begin{multline}
0 = \int_0^{2\pi} (bF)'= \int b'|u'|^2 + \mu^2 b'|u|^2 + 2\Re u'' \bar{u}' + \mu^2 2 \Re u' \bar{u}'\\
= \int b'|u'|^2 + \mu^2 b'|u|^2 - 2 b \Re f \bar{u}' - 2 \lambda b \Re i W u \bar{u}'. 
\end{multline}
Therefore
\begin{equation}
\int b' |u'|^2 + \mu^2 b' |u|^2 \leq 2 \int b |fu'| + 2 \lambda \int b W |uu'|.
\end{equation}
Now adding a multiple of \eqref{wuestimate} and \eqref{wuprime} to both sides gives
\begin{equation}
\int \phi |u'|^2 + \phi \mu^2 |u|^2 \leq 2 \int |fu'| + 2\lambda \int W|uu'| + \lambda^{-1} \int |fu| + C\left( 1+ \frac{\mu^2}{\lambda}\right) \int |fu|.
\end{equation}
Applying Young's inequality for products to the $|fu|$ terms on the right hand side, absorbing the resultant $|u|^2$ terms back into the left hand side and recalling $W=\sum W_j$ and $\mu^2 \leq \lambda^2$ gives the desired inequality. 
\end{proof}

It now remains to estimate the $W_j$ terms. We begin with an estimate in the case $M=\Sb^1$, so $\mu^2=\lambda^2$ and consider $W_j \in L^{\infty}$. Because $W_j$ satisfies hypotheses of the classical geometric control argument, one expects this argument to be straightforward and it is.
\begin{lemma}\label{linftycontrol}
When $W_j \in L^{\infty}$ and $M=\Sb^1$ for any $\e>0$ there exists $C>0$ such that if $\lambda>0, \mu^2 =\lambda^2$ and $u,f$ solve \eqref{SDWE}
\begin{equation}
\lambda \int W_j |uu'| \leq C \int |f|^2 + \frac{\e \mu^2}{2} \int |u|^2 + \frac{\e}{2} \int |u'|^2.
\end{equation}
\end{lemma}
\begin{proof}
Well, using that $W_j$ is bounded and \eqref{wuestimate}
\begin{align}
\lambda \int W_j |uu'| &\leq \frac{\lambda^2}{2\e} \int W_j^2 |u|^2 + \frac{\e}{2} \int |u'|^2 \leq \frac{C\lambda^2}{2 \e} \int W_j |u|^2 + \frac{\e}{2} \int |u'|^2 \\
&\leq \frac{C\lambda}{2\e} \int |fu| + \frac{\e}{2} \int |u'|^2 \\
&\leq \frac{C \lambda^2}{\mu^2} \int |f|^2 + \frac{\e \mu^2}{2} \int |u|^2 + \frac{\e}{2} \int |u'|^2. 
\end{align}
Finally, since $\mu^2= \lambda^2$ this gives the desired inequality. 
\end{proof}

We now turn to $W_j$ with polynomial type singularities. Following the structure of \cite{dk20} we prove an intermediate result and reduce the problem to estimating $V_j \chi_j |fu|$. In this proof we specify $\d_j=\frac{1}{2+\beta_j}$ in order to control the growth of a term. 
The proof of this lemma requires a change in technique from the proof in \cite{dk20} in order to account for the singularity and the fact that $W_j'$ is larger than $W_j$ near its singularity.

Let $\chi \in \Cs$ be supported on $|x|>1/2$ and be identically 1 on $|x|>1$. 
\begin{lemma}\label{resolvelemmamain}
If $W_j \in X^{\beta_j}_{\theta_j}$, then for any $\ti{\e}>0$, there exists $C>0$ and $\e>0$, such that if $\chi_j(x) = \chi\left(\frac{x-\sigma_j}{\e \lambda^{-\d_j}}\right)$ and if $\lambda, \mu^2>0$ and $u,f$ solve \eqref{SDWE}, then
\begin{equation}
\lambda \int W_j |uu'| \leq C (\lambda^{\d_j} + \mu) \int |fu| + C \lambda^{1/2}\left( \int |fu| \right)^{1/2} \left( \int V_j \chi_j |fu| \right)^{1/2} + \ti{\e} \int \phi |u'|^2. 
\end{equation}
\end{lemma}
\begin{proof}
Recall throughout that $\beta_j \in [-1,0)$. From $W_j \in X^{\beta_j}_{\theta_j}$ we have that there exists $C_j>0$ and $V_j=(|x|-\sigma_j)_+^{\beta_j}$ such that $\frac{1}{C_j} V_j(x) \leq W(x) \leq C_j V_j(x)$.

To begin make a change of variables so that $\sigma_j=0$ then $V_j=|x|^{\beta_j}$ and $\int W_j|uu'| \leq C \int V_j |uu'|$. 
The strategy is to split this integral into $|x|>\e \lambda^{-\d_j}$ and $|x|<\e \lambda^{-\d_j}$ where $\e>0$ is to be chosen later. 

Case 1: $|x|> \e \lambda^{-\d_j}$. Note $\chi_j$ as defined above is supported on $|x|>\frac{\e \lambda^{-\d_j}}{2}$ and is identically 1 on $|x|>\e \lambda^{-\d_j}$. So applying Cauchy-Schwarz and \eqref{wuestimate}
\begin{align}
\int_{|x|>\e \lambda^{-\d_j}} V_j |uu'| &\leq C \left(\int V_j |u|^2 \right)^{1/2} \left( \int V_j \chi_j |u'|  \right)^{1/2} \\
&\leq \lambda^{-1/2} \left( \int |fu| \right)^{1/2} \left( \int V_j \chi_j |u'|  \right)^{1/2}.
\end{align}
Using integration by parts 
\begin{equation}\label{eq:vjxbig}
\int V_j \chi_j |u'|^2 = - \Re \int (V_j \chi_j)' u' \bar{u} - \Re \int V_j \chi_j u'' \bar{u}.
\end{equation}
To control the first term note $|x| \geq \frac{\e \lambda^{-\d_j}}{2}$ on $\supp \chi_j$ so $\frac{1}{|x|} \leq \frac{2 \lambda^{\d_j}}{\e}$ there, also $|\chi'_j| \lesssim \frac{\lambda^{\d_j}}{\e}$ and so $ |(V_j \chi_j)'| \leq C \frac{\lambda^{\d_j} V_j}{\e}$. Therefore
\begin{equation}\label{eq:firstvjxbig}
\Re \int (V_j)' u' \bar{u} \leq C \frac{\lambda^{\d_j}}{\e} \int V_j |u u'|.
\end{equation}
For the second term apply \eqref{SDWE} and \eqref{wuestimate} to get 
\begin{align}\label{eq:secondvjxbig}
\Re \int V_j \chi u'' \bar{u}&= \Re \int V_j \chi (-i\lambda W u - \mu^2 u - f) \bar{u} \leq \Re \int V_j \chi \mu^2 |u|^2 + V_j \chi |fu| \nonumber \\
&\leq \Re \int \frac{\mu^2}{\lambda} |fu| + V_j \chi |fu|.
\end{align}
Combining \eqref{eq:firstvjxbig} and \eqref{eq:secondvjxbig} with \eqref{eq:vjxbig} gives 
\begin{equation}\label{eq:xbigfinal}
\lambda \int_{|x|>\e \lambda^{-\d_j} } V_j |uu'| \leq C \lambda^{1/2} \left( \int |fu| \right)^{1/2} \left( \frac{\lambda^{\d_j}}{\e} \int V_j |uu'| + \frac{\mu^2}{\lambda}|fu| + V_j \chi |fu| \right)^{1/2}.
\end{equation}

Case 2: $|x|<\e \lambda^{-\d_j}$. To begin note that 
\begin{equation}\label{mudj}
\phi \lambda^{-\d_j} = 1 \text{ on } |x|< \e \lambda^{-\d_j}< \lambda^{-\d_j}.
\end{equation}
So applying Cauchy-Schwarz
\begin{equation}\label{eq:xsmall}
\lambda \int_{|x| < \e \lambda^{-\d_j}} V_j |uu'| \leq \lambda^{1-\frac{\d_j}{2}} \left( \int_{|x|< \e\lambda^{-\d_j} } V_j^2 |u|^2 \right)^{1/2} \left( \int \phi |u'|^2 \right)^{1/2}.
\end{equation}
Now let $\psi$ be a cutoff supported in $|x|<2 \lambda^{-\d_j}$ and identically 1 on $|x|<\lambda^{-\d_j}$ then let $\psi_{\e}(x) = \psi(x/\e)$ and insert it into the below integral. Then rewriting $|x|^{2\beta_j}= -C \p_x |x|^{2\beta_j+1}$ and integrating by parts 
\begin{align}\label{eq:vjxsmall}
\int_{|x| <\e \lambda^{-\d_j} } V_j^2 |u|^2 &\leq \int |x|^{2 \beta_j} \psi_{\e} |u|^2 = - C \int \p_x\left(|x|^{2\beta_j+1} \right) \psi_{\e} |u|^2 \nonumber \\ 
&=C \int |x|^{2\beta_j+1} \p_x (\psi_{\e} |u|^2 ) \nonumber\\
&\leq C\int |x|^{2\beta_j+1} \psi_{\e}' |u|^2 + C \int |x|^{2 \beta_j+1} \psi_{\e} |uu'|.
\end{align}
To control the first term of \eqref{eq:vjxsmall} note 
\begin{align}\label{eq:firstvjxsmall}
\int_{|x| < 2 \e \lambda^{-\d_j} } |x|^{2\beta_j+1} \psi_{\e}' |u|^2 &= \int_{|x|<2 \e \lambda^{-\d_j}} |x|^{1+\beta_j} |x|^{\beta_j} \e^{-1} \psi' |u|^2 \nonumber \\
&\leq C \e^{\beta_j} \lambda^{-\beta_j\d_j} \int_{|x|<\e \lambda^{-\d_j}} V_j |u|^2 \nonumber \\
&\leq C \e^{\beta_j} \lambda^{-\beta_j\d_j} \int W |u|^2 \leq C \e^{\beta_j} \lambda^{-\beta_j \d_j -1} \int |fu|. 
\end{align}
Now let $\eta>0$ be a constant to be specified and applying Young's inequality for products to the second term of \eqref{eq:vjxsmall}, then using that $\psi_{\e}$ is supported on $\{|x| < 2 \e \lambda^{-\d_j}\}$ and \eqref{mudj} 
\begin{align}\label{eq:secondvjxsmall}
 \int |x|^{2\beta_j+1} \psi_{\e} |uu'| &\leq \frac{\eta}{2} \int |x|^{2\beta_j} \psi_{\e} |u|^2 + \frac{1}{2\eta} \int |x|^{2+2\beta_j} \psi_{\e} |u'|^2 \nonumber \\
&\leq  \frac{\eta}{2} \int |x|^{2\beta_j} \psi_{\e} |u|^2 + \frac{C}{\eta} \e^{2+2\beta_j} \lambda^{-2\beta_j \d_j -3\d_j} \int \phi |u'|^2.
\end{align}
So then combining \eqref{eq:vjxsmall}, \eqref{eq:firstvjxsmall} and \eqref{eq:secondvjxsmall}
\begin{align}
&\int \psi_{\e} V_j^2 |u|^2 \leq  C \e^{\beta_j} \lambda^{-\beta_j \d_j -1} \int |fu| + \eta C \int |x|^{2\beta_j} \psi_{\e} |u|^2 + \frac{C}{\eta} \e^{2+2\beta_j} \lambda^{-2\beta_j \d_j -3\d_j} \int \phi |u'|^2 \\
&\int \psi_{\e} V_j^2 |u|^2 \leq C \e^{\beta_j} \lambda^{-\beta_j \d_j -1} \int |fu| + C \e^{2+2\beta_j} \lambda^{-2\beta_j \d_j -3\d_j} \int \phi |u'|^2.
\end{align}
Where the second integral in the first inequality was absorbed back into the left hand side by choosing $\eta$ small enough. Combining this last inequality with \eqref{eq:xsmall}, letting $\e_0>0$ be a constant to be specified and applying Young's inequality for products
\begin{align}\label{eq:xsmallfinal}
\lambda \int_{|x|<\lambda^{-\d_j \e}} V_j |uu'| &\leq \lambda^{1-\d_j/2} \left( C \e^{\beta_j} \lambda^{-\beta_j \d_j-1} \int |fu| + C \e^{2+2\beta_j} \lambda^{-2\beta_j \d_j - 3\d_j} \int \phi |u'|^2 \right)^{1/2} \left( \int \phi |u'|^2 \right)^{1/2} \nonumber \\
&\leq \frac{\lambda^{2-\d_j}}{2 \e_0} \left( C \e^{\beta_j} \lambda^{-\beta_j\d_j -1} \int |fu| + C \e^{2+2\beta_j} \lambda^{-2\beta_j \d_j-3\d_j} \int \phi |u'|^2 \right) + \frac{\e_0}{2} \int \phi |u'|^2 \nonumber\\
&= C \e^{\beta_j} \lambda^{1-(\beta_j+1)\d_j} \int |fu| + \frac{C \e^{2+2\beta_j}}{\e_0} \lambda^{-(2\beta_j+4)\d_j +2} \int \phi |u'|^2 + \frac{\e_0}{2} \int \phi |u'|^2.
\end{align}
Now combining \eqref{eq:xbigfinal} and \eqref{eq:xsmallfinal} then applying Young's inequality for products to absorb the $V_j |uu'|$ term and a $\lambda^{1/2}$ from the right hand side back into the left hand side. 
\begin{align}
\lambda \int V_j |uu'| &\leq C \lambda^{1/2} \left( \int |fu| \right)^{1/2} \left( \frac{\lambda^{\d_j}}{\e} \int V_j |uu'| + \frac{\mu^2}{\lambda} |fu| + V_j \chi |fu| \right)^{1/2} \\
&+C \e^{\beta_j} \lambda^{1-(\beta_j+1)\d_j} \int |fu| + \frac{C\e^{2+2\beta_j} }{\e_0} \lambda^{-(2\beta_j+4)\d_j+2} \int \phi |u'|^2 + \frac{\e_0}{2} \int \phi |u'|^2\\
&\leq C(\lambda^{\d_j} + \mu ) \int |fu| + C \lambda^{1/2} \left(\int |fu| \right)^{1/2} \left( \int V_j \chi |fu| \right) + \left( \frac{C \e^{2+2\beta_j}}{\e_0} + \frac{\e_0}{2} \right) \int \phi |u'|^2.
\end{align}
Where the $\lambda$ dependence on the second to last $\phi |u'|^2$ term is eliminated by setting $\d_j=\frac{1}{\beta_j+2}$, 
as then $-(2\beta_j+4)\d_j+2=0$. This also ensures that $1-(\beta_j+1)\d_j=\d_j$. 

Choosing $\e_0$ small enough and then $\e$ small enough gives the desired inequality. 
\end{proof}
\begin{remark}\label{oneside}
Note that when $M=\Sb^1$, the estimate of $\int W_j |uu'|$ can be modified to have a third case without changing the result. That is for some small $c>0$, consider $c>|x|>\e \lambda^{-\d_j}, |x|<\e\lambda^{-\d_j}$ and $|x|>c$. The first two cases are proved as normal and the case $c>|x|$ can be controlled by Lemma \ref{linftycontrol}. This is what makes it possible to address $W_j=1_{[\pi,-\sigma]}(x) V_j$ or $1_{[\sigma,\pi)}(x)V_j$ for $V_j \in X^{\beta_j}_{\theta_j}$ when $M=\Sb^1$. 
\end{remark}

To obtain the desired resolvent estimate it remains to control the $V_j \chi_j |fu|$ term.
\begin{lemma}\label{resolvelemmafinal}
If $\lambda>0, \mu^2 \in \Rb$ and $u,f$ solve \eqref{SDWE} then
\begin{equation}
\lambda^{1/2} \left( \int |fu| \right)^{1/2} \left( \int V_j \chi_j |fu| \right)^{1/2} \leq C \int |f|^2.
\end{equation}
\end{lemma}
\begin{proof}
By linearity there are two cases
\begin{enumerate}
	\item $\supp f \subset (\supp V_j)^c$
	\item $\supp f \subset \supp V_j$
\end{enumerate}
In case 1 the term $\int V_j \chi_j |fu|$ vanishes. 

In case 2 note that $V_j \geq C$ on $\supp V_j$ so 
\begin{equation}
\int |fu| \leq c \int V_j^{1/2} |fu|.
\end{equation}
Then using Cauchy-Schwarz and \eqref{wuestimate} 
\begin{equation}
\int |fu| \leq c\left( \int |f|^2 \right)^{1/2} \left( \int V_j |u|^2 \right)^{1/2} \leq \frac{C}{\lambda^{1/2}} \left( \int |f|^2 \right)^{1/2} \left( \int |fu| \right)^{1/2}.
\end{equation}
Therefore $\int |fu| \leq \frac{C}{\lambda} \int |f|^2$. 
From this and Cauchy-Schwarz
\begin{equation}
\lambda^{1/2} \left( \int |fu| \right)^{1/2} \left( \int \chi_j V_j |fu| \right)^{1/2} \leq \left( \int |f|^2 \right)^{1/2} \left( \int |f|^2\right)^{1/4} \left( \int \chi_j V_j^2 |u|^2 \right)^{1/4}.
\end{equation}
It remains to control the final term on the right hand side. Because $\chi_j$ is supported on $x>\e \lambda^{-\d_j}$ and $V_j=|x|^{\beta_j}$.
\begin{equation}
\int \chi V_j^2 |u|^2 \leq C \lambda^{-\d_j \beta_j} \int V_j |u|^2 \leq C \lambda^{-\d_j \beta_j-1} \int |fu| \leq C \lambda^{-\d_j \beta_j -2} \int |f|^2 \leq C \int |f|^2.
\end{equation}
Where the final inequality holds because $\beta \geq -1 >- 4/3$ and $\d \beta = \frac{\beta}{\beta+2} \leq 2$ when $\beta \leq 2\beta+4$ so $-\d\beta-2 \leq 0$ . Therefore in case 2
\begin{equation}
\lambda^{1/2} \left( \int |fu| \right)^{1/2} \left( \int V_j \chi |fu| \right)^{1/2} \leq C \int |f|^2. 
\end{equation}
\end{proof}
Combining Lemmas \ref{resolvelemmamain} and \ref{resolvelemmafinal} and, when $M=\Sb^1$, Lemma \ref{linftycontrol}
\begin{equation}
\lambda \int W_j |u u'| \leq C \left( \lambda^{\d_j} + \mu \right) \int |fu| + C \int |f|^2 + \tilde{\e} \int \phi |u'|^2.
\end{equation}
This along with Lemma \ref{resolvelemmastart}
\begin{equation}
\int \phi |u'|^2 +\mu^2 \phi |u|^2 \leq C \sum_j\left( \lambda^{\d_j} + \mu \right) \int |fu| + C \int |f|^2 + \tilde{\e} \int \phi |u'|^2.
\end{equation}
Then since $\ti{\e}$ can be taken small enough to allow the $\phi|u'|^2$ on the right hand side to be absorbed into the left hand side
\begin{equation}
\int \phi |u'|^2 + \mu^2 \phi |u|^2 \leq C \left( \frac{\lambda^{\d_j}}{\mu^2} + \frac{1}{|\mu|}\right) \int |f|^2. 
\end{equation}
And so 
\begin{equation}
\int \frac{1}{\mu^2} |u'|^2 + \int |u|^2 \leq \frac{C}{\mu^2} \left( \frac{\lambda^{2\d}}{\mu^2}+1\right) \int |f|^2,
\end{equation}
which is exactly \eqref{eq:highenergy}.

\section{Sharpness of energy decay}\label{sharpnesssection}
Throughout this section we assume $\beta\in (-1,0)$. We will follow the strategy of \cite{kle19b} with some changes. We begin by constructing a sequence of solutions to a classical equation with a semiclassically small perturbation term. 
\begin{lemma}\label{4t1}
Fix $\beta\in (-1,0)$. There exists $h_0>0$ small such that for each $h\in (0,h_0)$ and $z\in\mathbb{C}$ with $\abs{z}\le \frac{1}{2} h^{-2}$, there exists $v_{h,z}\in H^1(0,\infty)$ such that 
\begin{equation}
Q_{h,z}v_{h,z}=\left(-\partial_y^2-i(y^\beta+h^{-\frac{2\beta}{2+\beta}})-h^{\frac{4}{2+\beta}}z\right)v_{h,z}=0,
\end{equation}
with $\partial_y v_{h,z}(0)=1$ and 
\begin{equation}
{C^{-1}}\le \|v_{h,z}\|_{H^1(0,\infty)}\le C h^{\frac{\beta}{2+\beta}}.
\end{equation}
Moreover $v_{h,z}(0)$ is bounded and analytic in $\{z\in \mathbb{C}: \abs{z}\le \frac{1}{2}h^{-2}\}$ for each fixed $h\in(0, h_0)$, and there exists $K>0$ such that
\begin{equation}
{C^{-1}h^{-\frac{2\beta}{2+\beta}}}\le\abs{v_{h,z}(0)}\le K.
\end{equation}
The constants are independent of $h$ and $\eta$.
\end{lemma}
\begin{proof}
1. We claim for those $h$ and $z$, $Q_{h,z}$ is an invertible operator from $H^1(0, \infty)$ to $H_0^{-1}(0, \infty)$, the dual space of $H^1(0, \infty)$. Consider the bounded bilinear form 
\begin{equation}
B(u,v)=\langle \partial_y u, \partial_y v\rangle-h^{\frac{4}{2+\beta}}z\langle u, v\rangle-i\left(\langle y^{\beta}u, v\rangle+h^{-\frac{2\beta}{2+\beta}}\langle u,v\rangle\right),
\end{equation}
on $H^1(0,\infty)$, as the variational formulation for $Q_{h,z}$ with Neumann conditions at $0$. For fixed $h>0$, the bilinear form is strongly coercive:
\begin{multline}
\abs{B(u,u)}^2=\left(\|\partial_y u\|^2-h^{\frac{4}{2+\beta}}\cre{z}\|u\|^2\right)^2+\left(\|y^{\frac{\beta}{2}} u\|^2+h^{-\frac{2\beta}{2+\beta}}\|u\|^2+h^{\frac{4}{2+\beta}}\cim{z}\|u\|^2\right)^2\\
\ge \frac{1}{2}\|\partial_y u\|^4-h^{\frac{8}{2+\beta}}\abs{z}^2\|u\|^2+\frac{1}{2}\|y^{\frac{\beta}{2}}u\|^4+\frac{1}{2}h^{-\frac{4\beta}{2+\beta}}\|u\|^2,
\end{multline}
where we used the algebraic inequality $(a-b)^2\ge \frac{1}{2}a^2-b^2$ for non-negative $a,b$. Note $h^{\frac{8}{2+\beta}}\abs{z}^2\le \frac{1}{4} h^{-\frac{4\beta}{2+\beta}}$ and we have
\begin{equation}
 \abs{B(u,u)}\ge \frac{1}{\sqrt{2}}\left(\frac{1}{\sqrt{2}}\|\partial_y u\|^4+\frac{1}{\sqrt{2}}\|y^{\frac{\beta}{2}}u\|^4+\frac{1}{2}h^{-\frac{2\beta}{2+\beta}}\|u\|^2\right)\ge \frac{1}{\sqrt{8}} h^{-\frac{2\beta}{2+\beta}}\|u\|^2_{H^1(0,\infty)},
\end{equation} 
the coercivity we need. By the Lax-Milgram theorem we know for any $f\in H_0^{-1}(0,\infty)$, there is a unique $v\in H^1(0,\infty)$ denoted by $v=R_{h, z}f$ such that $B(v,\cdot)=\langle f, \cdot\rangle$ as a functional on $H^1(0,\infty)$, that is, $Q_{h,z}v=f$ with $\partial_y v(0)=0$. We further note $R_{h,z}$ is a right inverse to $Q_{h,z}: H^1\rightarrow H^{-1}_0$ and
\begin{equation}\label{4l5}
C_1^{-1}\le \|R_{h,z}\|_{H_0^{-1}\rightarrow H^{1}}\le C h^{\frac{\beta}{2+\beta}},
\end{equation}
the upper and lower bounds are respectively from the Lax-Milgram theorem and the boundedness of $Q_{h,z}$. 

2. We claim that there exists $C>0$, such that for any $f\in y^{\frac{\beta}{2}}L^2(0, \infty)$ we have
\begin{equation}
\abs{R_{h,z}f(0)}\le C\|y^{-\frac{\beta}{2}}f\|. 
\end{equation}
Note that Lemma \ref{2t3} implies that $f\in H^{-1}$. Denote $v=R_{h,z}f$. Observe
\begin{equation}\label{4l1}
\langle f, v\rangle=B(v, v)=\|\partial_y v\|^2-h^{\frac{4}{2+\beta}}\cre{z}\|v\|^2-i\left(\|y^{\frac{\beta}{2}} v\|^2+h^{-\frac{2\beta}{2+\beta}}\|v\|^2+h^{\frac{4}{2+\beta}}\cim{z}\|v\|^2\right),
\end{equation}
the imaginary part of which implies
\begin{equation}
\|y^{\frac{\beta}{2}} v\|^2+h^{-\frac{2\beta}{2+\beta}}\|v\|^2+h^{\frac{4}{2+\beta}}\cim{z}\|v\|^2\le \abs{\langle f,v\rangle}\le C\epsilon^{-1}\|y^{-\frac{\beta}{2}}f\|^2+\epsilon \|y^{\frac{\beta}{2}}v\|^2. 
\end{equation}
Note that $\abs{z}\le \frac{1}{2}h^{-2}$ and absorb the last term on the right by the left, to obtain
\begin{equation}\label{4l2}
\|y^{\frac{\beta}{2}}v\|^2+h^{-\frac{2\beta}{2+\beta}}\|v\|^2\le C\epsilon^{-1}\|y^{-\frac{\beta}{2}}f\|^2. 
\end{equation}
Furthermore, since $y^{\frac{\beta}{2}}\ge 1$ on $(0,1]$, we have $\|v\|_{L^2(0,1)}^2\le C\|y^{-\frac{\beta}{2}}f\|$. The real part of \eqref{4l1} and \eqref{4l2} implies
\begin{equation}
\|\partial_y v\|^2\le C\|y^{-\frac{\beta}{2}}f\|^2+C\|y^{\frac{\beta}{2}}v\|^2+h^{\frac{4}{2+\beta}}\cre{z}\|v\|^2\le C\|y^{-\frac{\beta}{2}}f\|^2.
\end{equation}
Thus, by the trace theorem, $\abs{v(0)}\le \|v\|_{H^1(0,1)}\le C\|y^{-\frac{\beta}{2}}f\|$ as we need.

3. We claim $R_{h,z}f(0)$ is analytic in $z\in \{z\in \mathbb{C}: \abs{z}\le \frac{1}{2}h^{-2}\}$, at each fixed $h$. Fix $z_0\in\{\abs{z}\le \frac{1}{2}h^{-2}\}$. For any $z\in \{\abs{z-z_0}\le \|R_{h,z_0}\|^{-1}_{H^{-1}_{0}\rightarrow H^{1}}\}$, we have from the resolvent identity that
\begin{equation}
R_{h,z}=R_{h,z_0}\left(1-\left(z-z_0\right)R_{h,z_0}\right)^{-1}=\sum_{k=0}^{\infty}(z-z_0)^{k}R_{h,z_0}^{k+1},
\end{equation}
which converges in norm. Thus $R_{h,z}: H^{-1}_0\rightarrow H^1$ is analytic in $\{\abs{z}\le \frac{1}{2}h^{-2}\}$, and because the trace at $0$ is bounded by the $H^1$-norm, we have the desired analyticity. 

4. We now construct $v_{h,\eta}$ with specific Neumann data at 0. Fix a real $\chi\in C_c^{\infty}([0,\infty))$ with $\chi(0)=0$, $\partial_y\chi(0)=1$. We pick $\chi$ such that $\|y^\beta\chi\|\ge 2(C_1+1)\|\chi\|$, where $C_1$ is the bound in \eqref{4l5}. Then $Q_{h,z}\chi$ is smooth and compactly supported and
\begin{equation}
2C_1 \|\chi\|\le \|y^{\beta}\chi\|-h^{-\frac{2\beta}{2+\beta}}\|\chi\| \le \|Q_{h,z} \chi\|_{L^2}\le C\|\chi\|_{H^2}\le C_\chi.
\end{equation}
Now let $v_{h,z}=\chi-R_{h,z}Q_{h,z}\chi$. Immediately, we see $Q_{h,z} v_{h,z}=0$ and by the construction of $R_{h,z}$ in Step 1, $\p_y R_{h,z} Q_{h,z} \chi=0$ so $\p_y v_{h,z}(0)=1$. Now, note that since $R_{h,z}$ is bounded below as an operator
\begin{equation}
\hp{R_{h,z} Q_{h,z} \chi}{1} \geq \frac{1}{C_1} \ltwo{Q_{h,z} \chi} \geq 2 \nm{\chi}.
\end{equation}
And since $R_{h,z}$ is bounded above as an operator 
\begin{equation}
\hp{R_{h,z} Q_{h,z} \chi}{1} \leq C h^{\frac{\beta}{2+\beta}}. 
\end{equation}
Therefore
\begin{equation}
C \leq \hp{v_{h,z}}{1} \leq C h^{\frac{\beta}{2+\beta}}.
\end{equation}
Furthermore, note that $\nm{y^{-\frac{\beta}{2}} Q_{h,z} \chi} \leq C_{\chi}$, so by Step 2 and $\chi(0)=0, |v_{h,z}(0)| \leq K$.

5. It remains to show $\abs{v_{h,z}(0)}$ is bounded from below. Observe
\begin{multline}
0=\langle Q_{h,z} v_{h,z}, v_{h,z}\rangle=v_{h,z}(0) \partial_y v_{h,z}(0)+\|\partial_y v_{h,z}\|^2-h^{\frac{4}{2+\beta}}\cre{z}\|v_{h,z}\|^2\\
-i\left(\|y^{\frac{\beta}{2}} v_{h,z}\|^2+h^{-\frac{2\beta}{2+\beta}}\|v_{h,z}\|^2+h^{\frac{4}{2+\beta}}\cim{z}\|v_{h,z}\|^2\right),
\end{multline}
which after simplification becomes
\begin{equation}
\abs{v_{h,z}(0)}\ge C^{-1}\left(\|\partial_y v_{h,z}\|^2+\|y^{\frac{\beta}{2}} v_{h,z}\|^2+h^{-\frac{2\beta}{2+\beta}}\|v_{h,z}\|^2\right)\ge C^{-1}h^{-\frac{2\beta}{2+\beta}}\|v_{h,z}\|_{H^1}^2\ge C^{-1}h^{-\frac{2\beta}{2+\beta}},
\end{equation}
where we used $\partial_y v_{h,z}(0)=1$. 
\end{proof}
\begin{remark}
Note that the resolvent $R_{h,z}$ is not defined at $h=0$, which makes it different from the case of \cite[Lemma 4.3]{kle19b}. We use a different strategy below, rather than invoking the implicit function theorem at $h=0, z=0$, as $h\rightarrow 0$. 
\end{remark}
We now move to construct the solutions to a complex absorbing potential problem with a complex Coulomb potential. 
\begin{proposition}[Eigenmodes on the half-line]
Fix $\beta\in(-1,0)$ and let $V(x)=\mathbbm{1}_{x\ge 0}(x^{\beta}+1)$. There exist $h_0,C>0$ such that for all $h\in (0,h_0)$, there are $\eta_h\in\mathbb{C}$, $v_h\in H^1(-\frac{\pi}{2},\infty)$ with $v_h(-\frac{\pi}{2})=0$ and $\|v_h\|_{H^1(-\frac{\pi}{2},\infty)}\le C$ such that
\begin{equation}\label{4l3}
Q_h v_h(x)=\left(-h^2\partial_x^2-iV(x)-h^2\eta_h^2\right)v_h(x)=0.
\end{equation}
Furthermore
\begin{equation}
\abs{\eta_h-2}\le \frac{9K}{\pi} h^{\frac{2}{2+\beta}}, 
\end{equation}
where $K$ is the distinguished bound in Lemma \ref{4t1} and is independent of $h$.
\end{proposition}
\begin{proof}
1. We first solve the equation \eqref{4l3} on $(0,\infty)$. By rescaling via $y=h^{-\frac{2}{2+\beta}}x$, the equation on $(0,\infty)$ is reduced to 
\begin{equation}
Q_h v_h=h^{\frac{2\beta}{2+\beta}}Q_{h,\eta}v_{h}(y)=h^{\frac{2\beta}{2+\beta}}\left(-\partial_y^2-i(y^\beta+h^{-\frac{2\beta}{2+\beta}})-h^{\frac{4}{2+\beta}}\eta^2\right)v_h(y)=0,
\end{equation}
on $y\in (0,\infty)$, where $Q_{h,\eta}$ is $Q_{h,z}$ in Lemma \ref{4t1} with $z=\eta_h^2$. 
Apply Lemma \ref{4t1} to see that for any $h\in (0,h_0)$, $\eta\in \{\eta\in\mathbb{C}:\abs{\eta}\le\frac{1}{2}h^{-2}\}$, there exists a sequence of $v_{h,\eta}\in H^1(0,\infty)$ that $Q_{h,\eta}v_{h,\eta}(y)=0$ with $v_{h,\eta}(0)$ is analytic in $\eta$ and
\begin{equation}
\abs{v_{h,\eta}(0)}\le K, \ \partial_y v_{h,\eta}(0)=1.
\end{equation}
Let $v^+_{h,\eta}(x)=v_{h,\eta}(h^{-\frac{2}{2+\beta}}x)\in H^1(0,\infty)$ denote such solutions in $x\in (0,\infty)$. Note that $Q_{h}v^{+}_{h,\eta}=0$ on $(0,\infty)$, $\|v_{h,\eta}^+\|_{H^1}\le C h^{\frac{-1+\beta}{2+\beta}}$ and $v^+_{h,\eta}(0)$ is analytic in $\eta$ and
\begin{equation}
\abs{v^+_{h,\eta}(0)}\le K, \ \partial_x v^+_{h,\eta}(0)=h^{-\frac{2}{2+\beta}}.
\end{equation}
This is our solution to \eqref{4l3} corresponding to parameters $h,\eta$ on the positive half real line $(0,\infty)$. 

2. We then solve the equation \eqref{4l3} on $(-\frac{\pi}{2},0)$. On $(-\frac{\pi}{2},0)$ the potential $W(x)\equiv 0$ and the equation \eqref{4l3} is solved by $v^-_{h,\eta}=\sin(\eta(x+\frac{\pi}{2}))$. We take note of the Cauchy data at $x=0$:
\begin{equation}
v^-_{h,\eta}(0)=\sin(\frac{\eta\pi}{2}), \ \partial_x v^-_{h,\eta}(0)=\eta \cos(\frac{\eta\pi}{2}).
\end{equation}

3. We now match $v^\pm_{h,\eta}$ at $0$. The transition condition is now
\begin{equation}
\frac{v^-_{h,\eta}(0)}{\partial_x v^-_{h,\eta}(0)}=\frac{v^+_{h,\eta}(0)}{\partial_x v^+_{h,\eta}(0)},
\end{equation}
which is
\begin{equation}\label{4l4}
\tan(\frac{\eta \pi}{2})=h^{\frac{2}{2+\beta}}\eta v_{h,\eta}^+(0).
\end{equation}
We will be looking for solutions $\eta=2+h^{\frac{2}{2+\beta}}\mu$ with parameter $\mu$ in $\{\mu\in\mathbb{C}:\abs{\mu}\le \frac{9K}{\pi}\}$. Let $v^+_{h,\mu}=v^+_{h,\eta(\mu)}$ be parametrised in $\mu$. Consider the objective function
\begin{equation}
F_h(\mu)=(2+h^{\frac{2}{2+\beta}}\mu)v^{+}_{h,\mu}(0)-h^{-\frac{2}{2+\beta}}\tan(\pi(2+h^{\frac{2}{2+\beta}}\mu)/2),
\end{equation}
the zeros of which solve \eqref{4l4}. We now assume $h_0$ is small such that $v^+_{h,\mu}(0)$ and $F_h(\mu)$ are analytic in $\{\abs{\mu}\le \frac{9K}{\pi}\}$, at each $h\in(0,h_0)$. Note
\begin{equation}
\tan(\pi(2+h^{\frac{2}{2+\beta}}\mu)/2)=\frac{1}2\pi h^{\frac{2}{2+\beta}}\mu+\bigo{(h^{\frac{6}{2+\beta}}\mu^3)},
\end{equation}
and since $v^+_{h,\mu}(0)$ and $\mu$ are bounded by $h$-independent constants, we have
\begin{equation}
F_h(\mu)=2v^+_{h,\mu}(0)-\frac12\pi\mu+\bigo(h^{\frac{2}{2+\beta}}\mu).
\end{equation}
Assume against contradiction that $F_h(\mu)$ does not have a zero in $\{\abs{\mu}\le \frac{9K}{\pi}\}$. Since $F_h(\mu)$ is analytic in the disk, by the minimum modulus principle, $\abs{F_h(\mu)}$ achieves its minimum over $\{\abs{\mu}\le \frac{9K}{\pi}\}$ on its boundary. Consider any $\abs{\mu}=\frac{9K}{\pi}$, we have when $h$ is small that
\begin{equation}
\abs{F_h(\mu)}\ge \frac{1}{2}\pi\abs{\mu}-2\abs{v^+_{h,\mu}(0)}+\bigo(h^{\frac{2}{2+\beta}}\abs{\mu})\ge \frac94 K. 
\end{equation}
However, $F_h(0)=v^+_{h,0}(0)$ and $\abs{F_h(0)}\le 2K \leq \frac{9}{4} K$, which is a contradiction. Thus for any $h\in (0,h_0)$, there exists $\mu_h$ in $\{\mu\in\mathbb{C}:\abs{\mu}\le \frac{9K}{\pi}\}$ such that $F_h(\mu_h)=0$. Then $\eta_h=2+h^{\frac{2}{2+\beta}}\mu_h$ satisfies the transition condition \eqref{4l4} at $0$, and $Q_h v_h=0$ on $(-\pi,\infty)$ with
\begin{equation}
v_h=v_{h,\eta_h}^--\beta_h v_{h,\eta_h}^+, \ \beta_h=2h^{\frac{2}{2+\beta}}+\mu_h h^{\frac{4}{2+\beta}}+\bigo(h^{\frac{6}{2+\beta}}), 
\end{equation}
and $\|v_h\|_{H^1}\sim C$ in $h\in (0, h_0)$. 
\end{proof}
\begin{remark}
Those eigenfunctions have most of its $H^1$-mass in $\{V=0\}$, and penetrates into $\{V>0\}$ with the Dirichlet data $v_h(0)=\bigo(h^{\frac{2}{2+\beta}})$ at the boundary between those two regimes. The eigenvalues are $h^{\frac{2}{2+\beta}}$-sized perturbation from the Dirichlet eigenvalues $2$ corresponding to those eigenfunctions that vanish at $-\frac{\pi}{2}$ and $0$. When $\beta\rightarrow 0^-$, this is consistent with the observation made by Nonnenmacher in \cite{aln14} in the limit case. 
\end{remark}
\begin{proposition}[Low-frequency quasimodes on $\mathbb{S}^1$]\label{4t2}
Parametrize $\mathbb{S}^1$ by $(-\pi,\pi)$ and let $W(x)=\mathbbm{1}_{\abs{x}\ge \frac{\pi}{2}}\left((\abs{x}-\frac{\pi}{2})^\beta+1\right)$. Then for any sequence $\{\lambda_k\}\subset\mathbb{R}$, $\lambda_k\rightarrow\infty$, there exists a sequence of $u_k\in H^1(\mathbb{S}^1)$, $\|u_k\|_{L^2}\equiv 1$ such that
\begin{equation}
\left(-\partial_x^2-i\lambda_k W(x)-2\right)u_k=\bigo_{L^2}(\lambda_k^{-\frac{1}{2+\beta}}). 
\end{equation}
\end{proposition}
\begin{proof}
Consider a cutoff $\chi\in C^\infty[-\frac{\pi}{2},\frac{\pi}{2}]$ with $\chi\equiv 1$ on $[-\frac{\pi}{2},\frac{\pi}{4}]$, and $\chi\equiv 0$ on $[\frac{3\pi}{8},\frac{\pi}{2}]$. Note that $[Q_h, \chi]\in \Psi_h^1$ has semiclassical microsupport inside $T^*((\frac{\pi}{4},\frac{3\pi}{8}))$, on which $Q_h$ is semiclassically elliptic. For $(x,\xi)\in T^*((\frac{\pi}{4},\frac{3\pi}{8}))$
\begin{equation}
\abs{\sigma_h(Q_h)}=\abs{\xi^2-iV}\ge \abs{\xi}^2+1 \ge 1.
\end{equation}
Therefore, by the elliptic estimate in \cite[Theorem E.33]{dz19}, for any $N\in \mathbb{N}$, there exists $C>0$, dependent of $N$, such that for any $h\in (0,h_0)$
\begin{equation}
\|[Q_h, \chi]v_h\|_{H_h^1}\le C\|Q_hv_h\|_{L^2}+Ch^{N}\|v_h\|_{H^{-N}_h}\le Ch^{N}\|v_h\|_{L^2}\le Ch^{N},
\end{equation}
and hereby $\|Q_h \chi v_h\|_{L^2}=\bigo(h^{N})$. Let $h_k=\lambda_k^{-\frac{1}{2}}\rightarrow0$ and 
\begin{equation}
u_k(x)=\sgn(x)(\chi v_{h_k})(\abs{x}-\frac{\pi}{2})/\left\|\sgn(x)(\chi v_{h_k})(\abs{x}-\frac{\pi}{2})\right\|_{L^2(-\pi,\pi)}. 
\end{equation}
Then $u_k\in H^1(\mathbb{S}^1)$ are $L^2$-normalized, vanish near $\pm \pi$ and
\begin{equation}
\left(-\partial_x^2-i\lambda_k W(x)-(2+\mu_k)^2\right)u_k=\bigo_{L^2}(\lambda_k^{-\frac{N}{2}+1}),
\end{equation}
where $\mu_k=\mu_{h_k}=\bigo(\lambda_k^{-\frac{1}{2+\beta}})$. Pick $N$ large such that
\begin{equation}
\left(-\partial_x^2-i\lambda_k W(x)-2\right)u_k=\bigo_{L^2}(\lambda_k^{-\frac{1}{2+\beta}}). 
\end{equation}
So these $u_k$ are appropriate quasimodes. 
\end{proof}
\begin{proof}[Proof of Theorem \ref{sharpnessthm}]
In Proposition \ref{4t2}, we can pick a sequence $\lambda_k=(2+k^2)^{\frac{1}{2}}\rightarrow \infty$, 
\begin{equation}
w_k(x,y)=u_k(x)\sin(k y).
\end{equation}
Then
\begin{equation}
P_{\lambda_k}w_k=\left(-\partial_x^2-\partial_y^2 - i \lambda_k W(x) - \lambda_k^2\right)w_k=\bigo_{L^2_x}(\lambda_k^{-\frac{1}{2+\beta}})\sin(ky). =\bigo_{L^2(\mathbb{T}^2)}(\lambda_k^{-\frac{1}{2+\beta}}). 
\end{equation}
and this means that for any $\epsilon>0$, 
\begin{equation}
\|P_{\lambda_k}w_k\|_{L^2(\mathbb{T}^2)}= o(\lambda_k^{-\frac{1}{2+\beta}+\epsilon})\|w_k\|_{L^2(\mathbb{T}^2)}. 
\end{equation}
By Proposition \ref{polyprop}, the solution to the damped wave equation \eqref{DWE} cannot be stable at the rate of $\langle t\rangle^{-\frac{\beta+2}{\beta+3}-\epsilon}$ for any $\epsilon>0$. 
\end{proof}

\bibliographystyle{alpha}
\bibliography{mybib,Robib}

\end{document}